\documentclass{article}

\usepackage{amsthm}

\usepackage{type1cm}        
\usepackage{makeidx}         
\usepackage{graphicx}        
\usepackage{multicol}        
\usepackage[bottom]{footmisc}

\usepackage{newtxtext}       %
\usepackage{newtxmath}       

\usepackage{xypic} 

\newtheorem{theorem}{Theorem}
\newtheorem{lemma}[theorem]{Lemma}
\newtheorem{proposition}[theorem]{Proposition}
\newtheorem{corollary}[theorem]{Corollary}
\newtheorem*{conjecture}
{Conjecture}
\newtheorem{remark}[theorem]{Remark}

\newtheorem{definition}[theorem]{Definition}

\xyoption{all} 

\makeindex             

\newcommand{\C}{{\mathcal C}}
\newcommand{\Hi}{{\mathcal H}}

\newcommand{\N}{{\mathbb N}}

\newcommand{ \monus }{\stackrel{\bullet}{-}}

\newcommand{\eafw}
{{\tiny\begin{array}{cc}{\mathcal E} & \! \! {\mathcal A} \\[-1.5pt]   {\mathcal F} & \! \! {\mathcal W}\end{array}}}

\newcommand{\decode}{\mbox{\scalebox{1.2}{$\vartriangleright$}}}
\newcommand{\code}{\mbox{\scalebox{1.2}{$\vartriangleleft$}}}

\newcommand{\dc}{{}}

\newcommand{\ud}{{}}

\newcommand{\app}{\mbox{\scalebox{1.1}{$\vartriangleright$}\hspace{-0.6em}$|$\hspace{0.6em}}}
\newcommand{\lam}{\mbox{\scalebox{1.1}{$\vartriangleleft$}\hspace{-0.4em}$|$\hspace{0.4em}}}

\begin{document}

\title{On strict extensional reflexivity \\ in compact closed categories}
\author{Peter Hines}

\maketitle

\abstract{ {\small This article has two related aims.  The first is to study the categorical setting of Abramsky, Haghverdi, \& Scott's untyped linear combinatory algebras \cite{AHS}, and the second is to relate this to the much more recent work by Abramsky \& Heunen on Frobenius algebras in the infinitary setting \cite{AH}.  
The key to this is (extensional) reflexivity (i.e. the property of an object being isomorphic to its own internal hom $R \cong [R\rightarrow R]$).  We first characterise extensionally reflexive objects in compact closed categories, then consider when \& how this property may be `strictified' -- how we may give a monoidally equivalent category where the isomorphisms exhibiting reflexivity are in fact identity arrows. This results in small two-object compact closed categories consisting of a unit object and a single (non-unit) strictly reflexive object. \\
We then move on to studying the endomorphism monoids of such objects from an algebraic rather than logical or categorical viewpoint. We demonstrate that these  necessarily contain an interesting inverse monoid that may be thought of as Richard Thompson's iconic group $\mathcal F$ together with the equally iconic bicyclic monoid $\mathcal B$ of semigroup theory, with non-trivial interactions between the two derived from the Frobenius algebra identity -- and claim this as a particularly significant example of the (unitless) Frobenius algebras of \cite{AH}.
\\
We first develop the theory from a purely theoretical point of view, then move on to develop concrete examples, based on the algebra and category theory behind \cite{GOI1,GOI2,AHS}. The concrete examples we give are based on the traced monoidal category of partial injections, and reflexive objects in the compact closed category that results from applying the ${\bf Int}$ or ${\bf GoI}$ construction.  We then give compact closed categories, monoidally equivalent to compact closed subcategories of $\bf Int(pInj)$, where this reflexivity is exhibited by identity arrows, and show how the above algebraic structures (Thompson's $\mathcal F$, the bicyclic monoid, and Frobenius algebras) arise in a fundamental manner. } 
}

\begin{center}{\bf \em This work is of course !( ) 
dedicated to Samson Abramsky.}
\end{center}

\section{Historical background}
The starting point for this chapter is Girard's Geometry of Interaction program   -- in particular the first two parts.  It is by now well-established that compact closed categories model key aspects of this (see \cite{AHS,HS05} for a good account), and this observation motivated the name of Abramsky's categorical $\bf GoI$ construction \cite{SA96} (see Section \ref{IntGoI-sect}).

However, this immediately poses an interesting puzzle.  In \cite{GOI1}, Girard makes the rather cryptic comment  that his system `forgets types', even though the stated aim was to produce a model of the polymorphically typed System $\mathcal F$, rather than a purely  untyped system. The explanation seems to be that the GoI system moves from a rigidly typed system to an entirely untyped system, in order to build a more flexible (polymorphic) type system on top of this\footnote{This is commonly studied via has become known as Hyland-Ong types  \cite{HO}. E. Haghverdi \& P. Scott also presented a Geometry of Interaction system that is decidedly typed in \cite{HS05}; however, this current chapter concentrates on the purely untyped aspects of the Geometry of Interaction program, and (re-)introducing types would be a non-trivial subsequent step.}. 
The claim that there is an untyped logical system at the core of Girard's GOI  was borne out in  Abramsky, Haghverdi, and Scott's paper \cite{AHS} that gave  an untyped combinatory logic (precisely, linear-exponential combinatory logic) based on the primitives from Girard's first two Geometry of Interaction papers \cite{GOI1,GOI2}. 

It is natural to wish to study this type-freeness via the `objects as types' paradigm of categorical logic \cite{LS}.  In the classical / Cartesian world, Lambek \& Scott introduced, as models of untyped lambda calculus, the $C$-monoids of \cite{LS}, which may reasonably be viewed as monoids satisfying all of the axioms for Cartesian closure apart from the existence of a terminal object\footnote{The complete situation is slightly more subtle than this description suggests. There is a good case that C-monoids do indeed account for units, but these are implicit  --  a passage to the Karoubi envelope then makes them explicit. This is beyond the scope of this paper, but covered in \cite{LS}.\label{cmonunits}}.

Based on similar ideas, analogues of 
compact closure within monoids were studied in \cite{PHD,TAC}  (again motivated by J.-Y. Girard's first three Geometry of Interaction papers \cite{GOI1,GOI2,GOI3}) and the claim was made that certain  monoids derived from the GOI program satisfy, `unitless analogues of compact closure'.   This work was carried out independently of \cite{AHS}; however, the underlying algebraic structures -- based on Girard's original system -- are identical. 

There are significant subtleties associated with this claim\footnote{To complicate matters,  the axiomatisation in \cite{TAC} work was intended \& stated to be equivalent to that of \cite{PHD}. Unfortunately, as recently pointed out by C. Heunen, several severe typo.s made their way into the final published version -- this axiomatisation, at least, unambiguously does not capture compact closure!}.  With Cartesian closure, we are left with the essential concepts even without the terminal object.   The same -- at least for the usual axiomatisation -- does not apply to compact closure. 

\section{Compact closure}\label{CC-sect}
In \cite{KL}, the abstract 2-categorical definition of a compact closed category is shown to have a concrete characterisation in terms of the existence of a duality and distinguished arrows. It is by now standard to take this as fundamental. 

\begin{definition}\label{CC-def}
A symmetric monoidal category $(\C,\otimes,\sigma_{\_ ,\_},I)$ is {\bf compact closed} when it is equipped with : 
\begin{itemize}
\item  a {\bf dual} --  a contravariant monoidal functor  $(\_  )^* : \C^{op}\rightarrow \C$ satisfying $\left( ( \_ )^* \right)^* =Id_{\mathcal C}$ 
\item for all objects $A\in Ob(\mathcal C)$, distinguished {\bf unit} \& {\bf co-unit} arrows $\eta_A : I \rightarrow A \otimes A^*$ and $\epsilon_A : A^* \otimes A \rightarrow I$ that satisfy the {\bf yanking axiom} 
\[ (1_A\otimes \epsilon)(\eta\otimes 1_A)  \ =\ 1_A\  =\   (\epsilon_{A^*} \otimes 1_A)(1_A \otimes \eta_{A^*}) \]  
\end{itemize}
Using the usual diagrammatic conventions  from \cite{JS1,JS2}, the unit / co-unit arrows are drawn as `cups' and `caps'
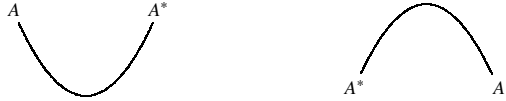
\begin{center}
\scalebox{0.8}{\small
\xymatrix{
	A \ar@/_40pt/@{}[rr]	&														& 	A^*	&\ \ \ \ \  \ \ \ \ \ &	 				& 									&		\\
		&														&		&\ \ \ \ \  \ \ \ \ \ & A^*\ar@/^40pt/@{}[rr]		&									& A 
} 
}
\end{center}
giving the yanking axiom as 
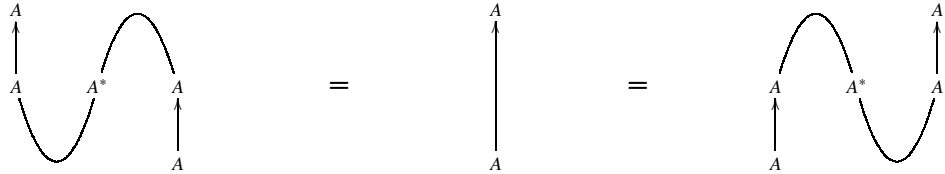
\begin{center}
\scalebox{0.8}{\small
\xymatrix{
A		&													&						&													&	A						&												&					& 								&	A \\
A\ar[u]	&	A^*  \ar@/^35pt/@{}[l]									& 	A \ar@/_35pt/@{}[l]		& \hspace{4em}		\mbox{ \Large =}	\hspace{4em}			&							&	\hspace{3em}		\mbox{ \Large =}	\hspace{3em}	& A\ar@/^35pt/@{}[r]	& A^* \ar@/_35pt/@{}[r]				& A\ar[u]		 \\	
		& 														&	A\ar[u]				&									 				&	A\ar[uu]					&		 										& A\ar[u]				&								&		\\
} 
}
\end{center}
\end{definition}

\begin{remark}[Interpretations and Examples]
The neat diagrammatic representation of the unit and co-unit arrows, and yanking, readily lead to numerous distinct interpretations. One of the earliest and most natural was for the unit and co-unit arrows to be interpreted as the {\bf Axiom} and {\bf Cut} rules of various forms of linear logic 
\begin{center}
\begin{tabular}{ccc}
Axiom : \ {\Large $\frac{}{A \ \ \ A^\perp}$}
& \hspace{6em}
Cut : \ {\Large $\frac{A^\perp \ \ \ A}{}$}
\\
\end{tabular}
\end{center}
This is the basis for the close connection between compact closure and the Geometry of Interaction discussed throughout, and the interpretation of yanking as cut-elimination is intuitively natural. The concrete examples we develop in Section \ref{concrete-sect} onwards are based on Girard's GoI system, and derived from the Int or GoI construction described in Section \ref{IntGoI-sect}.

An alternative interpretation was found in \cite{MSCS1,TCS1}, where algebraic models of  dynamics of Turing machines were observed to be compact closed;  the unit and co-unit model the changes of direction of the read-write head of a Turing machine as it moves over the tape, and the dual corresponds to interchanging the role of `left' and `right' in the definition of the dynamics. 

A further interpretation may be found in the Categorical Quantum Mechanics program of Abramsky \& Coeke \cite{AC}, where the unit and co-unit respectively correspond to the production of a maximally entangled state, and a (post-selected) measurement that results in the observation of this maximally entangled state. 

Numerous other interpretations have been given -- in particular, the pregroups of J. Lambek \cite{JL} arise from dropping the requirement of symmetry (implicitly, considering non-commutative multiplicative linear logic) to give linguistic models with a distinctly logical flavour.  See \cite{WB,WB03,JL} for the original motivation from linear logic and the identification as `non-symmetric compact closure'. 
\end{remark}

\subsection{On the definition(s) of compact closed monoids}\label{ccmprobs-sect}
It is not immediate how a compact closed monoid should be defined, apart from in the highly degenerate case where the unique object is the unit object. In this setting, we recall the folklore that the subcategory of {\em any} monoidal category generated by its unit object is compact closed, with MacLane's distinguished unit arrows \& their inverses trivially satisfying the axioms for the unit / co-unit of compact closure.

With Cartesian closure, we are generally happy to accept that the existence of a terminal object is a relatively minor part of the definition (and indeed may even be hidden in the structure of C-monoids -- see Footnote \ref{cmonunits}), and consider that a Cartesian closed monoid is one satisfying all the other parts of the definition.  By contrast, erasing the unit  object from the definition of compact closure leaves us with nothing -- not even the duality (some authors, notably \cite{JSV}, define the duality in terms of the unit / co-unit arrows).

 The approach taken in \cite{PHD,TAC} was to give an axiomatisation of compact closure for semi-monoidal categories (i.e. categories satisfying all MacLane's axioms except for those relating to the unit object -- see Definition \ref{semi-mon-def}) that 
\begin{enumerate}
\item does not mention the unit, and 
\item is equivalent to to the usual definition in the presence of a unit.
\end{enumerate}
Regardless of whether the logical interpretation of working without a unit is desirable, a fundamental question needs to be asked : 

\begin{quotation}{\bf  How do we know this definition of {\em unitless compact closure} is the `correct' definition?  Is it sufficient for it to  reduce to the standard definition in the presence of a unit?}
\end{quotation}
Expanding on this question, we may  ask what the implications would be of two such axiomatisations of unitless compact closure that both reduce to the usual definition in the presence of a unit, but are provably inequivalent?   How should we decide between them?

Such an pair of provably distinct axiomatisations has recently been established in joint (currently unpublished) work by the author and C. Heunen.    The concrete examples  of \cite{PHD,TAC} satisfy both sets of axioms, and the defining arrows of these axiomatisations correspond in each case to distinct interesting logical, computational, or categorical  features that we would be unhappy to exclude in the general case; there is simply no reason to prefer one axiomatisation over the other, and every reason to consider structures satisfying both axiomatisations.

The existence of concrete examples implies these are compatible, so may follow from a single set of axioms.    However, this illustrates that reduction to the usual axioms in the presence of a unit is therefore not sufficient. What is also needed is some notion of completeness   --  that any other compatible axiom scheme that is equivalent to compact closure in the presence of a unit is a consequence of it.  

The axioms given in \cite{PHD} do not have such a `completeness' property and, at best, cannot be the whole story.  The project described above remains ongoing. 

\section{From closed monoids to reflexive objects}
Nothing in the above discussion precludes compact closed categories where all non-unit objects are isomorphic (concrete examples given in part 4. of Proposition \ref{pNatresults-prop}), or even (in a small category) identical -- giving two-object compact closed categories, with a single non-unit object: see Corollary \ref{twoobjects-corol} and the concrete examples of Definition \ref{G-def}.  However, it is also worthwhile to take a step back and again ask the obvious question,

\begin{quotation}{\bf  ``What was the purpose of Lambek \& Scott introducing Cartesian Closed Monoids?''}
\end{quotation}

The immediate answer of course is, `to model untyped lambda calculus'. Looking slightly deeper  \cite{LS} observes that Cartesian closure, combined with the fact that there is only one object, means that the unique object $c$ is necessarily isomorphic to its own function space $c^c$ (i.e. is a {\em reflexive object}), and this is the key \cite{DS80} to the unrestricted application \& abstraction. 

\begin{remark}[Untyped systems -- a change of emphasis] From a categorical logic `objects as types' perspective, it seems natural that an `untyped version of $\mathcal X$' should be modeled by an ``$\mathcal X$--monoid'', provided this can be defined in a satisfactory manner.  When the intention is to model the unrestricted application / abstraction of an untyped lambda calculus, from precisely the same perspective the notion of reflexivity is the desirable property.  These two concepts -- although closely related -- are not identical.
\end{remark}

We therefore take as fundamental the notion of objects that are `isomorphic to their own internal hom'. 
The most general setting in which reflexivity may be defined is that of the `non-monoidal closed categories' of M. Laplaza \cite{ML77} (see Remark \ref{nonmon-rem} for the justification for such a general setting).  These are defined simply as categories with an internal hom functor (i.e. without explicit reference to any monoidal tensor or adjunction between hom and tensor) that satisfies some fairly intricate coherence conditions. Laplaza's definition, of course, includes the more familiar monoidal closed categories as a special case.

The following is taken from \cite{DS80} (see also \cite{MH17}), although we make some necessary (see Remark \ref{strict-rem}) changes in terminology.
\begin{definition}
Let $(\C , [ \_ \rightarrow \_])$ be a closed category, in the sense of \cite{ML77} (this includes monoidal closed categories as a special case). 

An {\bf (extensionally) reflexive object}, or simply {\bf reflexive object} $R\in Ob(\mathcal C)$ is one that is isomorphic to its own internal hom, so $R\cong[R\rightarrow R]$.

Explicitly, reflexive objects are equipped with mutually inverse isomorphisms :
\begin{itemize}
\item The {\bf app} isomorphism $\app   : R\rightarrow  [R\rightarrow R]$
\item The {\bf lam} isomorphism $\lam  : [R\rightarrow R]\rightarrow  R$ 
\end{itemize}
satisfying  $\lam\app=1_R$ and $\app\lam=1_{[R\rightarrow R]}$. 

We say that $R$ is {\bf weakly} or {\bf intensionally reflexive} when $[R\rightarrow R]$ is a retract of $R$, so $\app\lam=1_{[R\rightarrow  R]}$, but $\lam\app=e^2=e$ is an idempotent of $\C(R,R)$.  Note that this breaks with convention somewhat, by allowing for extensional reflexivity to be a very special case of weak reflexivity.

Finally, we say a reflexive object is {\bf strictly  reflexive} when $\lam$ and $\app$ are identity arrows. Strictly reflexive objects are of course extensional.
\end{definition}
\begin{remark}\label{strict-rem}[A conflict of terminology]
It is more common (e.g. \cite{MH17,DS80}) to refer to (extensional) reflexivity as {\em ``strict reflexivity''}, and to what we call ``weak reflexivity'' simply as {\em ``reflexivity''}.   
This immediately sets us up for a strong and probably unavoidable conflict of notation;  we are interested in `strictifying' reflexivity in the sense of strictification within categorical coherence -- i.e. giving a suitable equivalence of categories under which the isomorphisms exhibiting (extensional) reflexivity become identity arrows. \\

The usage of the term ``strict'' is very well-established in both fields, so this conflict of notation is unavoidable -- all we can do is point out the conventions we use!
\end{remark}

\begin{remark}[Weak, strong, \& strict reflexivity in models of combinatory logic]\label{lcamodels-rem}
A key aim of this chapter is to study the categorical / algebraic setting for the untyped combinatory logic of Abramsky, Haghverdi \& Scott, which we claim is that of strict reflexivity.  In some sense, compact closed monoids would be too extravagant a setting. By contrast to the logical interpretation of \cite{GOI1,GOI2}, the system of \cite{AHS} has no connectives (\& hence no need for  tensors), and no negation  (\& hence no non-degenerate duality is required)\footnote{However, units do play a r\^ole in \cite{AHS}, so reflexivity within  a setting including these is essential.};  in a $\lambda$ calculus, the key notions are those of application \&  abstraction.   From that viewpoint, it is appropriate to concentrate on reflexivity.   

However, combinatory logics are `lower-level' systems than $\lambda$ calculii.  The notion of abstraction itself is derived from more primitive operations --  an operation that acts like lambda-abstraction is built up using combinators.  Our claim is that in models of untyped combinatory logic, we will not only observe reflexivity, but reflexivity that is {\em strict} --- the isomorphisms that exhibit reflexivity are identities.

Finally, although reflexivity already features heavily in \cite{AHS}, it is intensional, or weak, reflexivity.  Girard's original system (implicitly) considered both -- we are therefore `plugging a gap'. It is also difficult to see how reflexivity that is not exhibited by isomorphisms may be strictified, unless we pass to the Karoubi envelope (i.e. idempotent splitting). This will then provide reflexive objects that are extensionally reflexive --- the setting we consider.  
\end{remark}

\subsection{Reflexive objects of compact closed categories}
Reflexive objects of compact closed categories have a particularly simple characterisation.

\begin{definition}
Let $(\mathcal C , \_ \otimes \_,I)$ be a monoidal category. An object $N\in Ob({\mathcal C})$ is called {\bf self-similar} or {\bf pseudo-idempotent} when it satisfies $N \cong N\otimes N$. The isomorphisms exhibiting this self-similarity are   unique up to unique isomorphism \cite{JHRS}, and commonly referred to as the {\bf code} and {\bf decode} arrows $\code\in {\mathcal C}(N\otimes N,N)$ and $\decode \in \mathcal C (N, N \otimes N)$ respectively.  

Similarly, let $(\mathcal D, ( \_ )^*)$ be a category with a dual. An object $S\in Ob({\mathcal D})$ is called {\bf self-dual} when it satisfies $S \cong S^*$; there is no standard notation or terminology for arrows exhibiting self-duality.   
\end{definition}

Examples of self-similar objects are given in Section \ref{Nselfsim-sect}, and self-similar objects of compact closed categories in Lemma \ref{NNselfsim-lem}, Proposition \ref{pNatresults-prop}, and Definition \ref{G-def}.  A characterisation of self-dual objects in a large class of compact closed categories is given in Corollary \ref{selfdualdagger-corol}, with the example of Lemma \ref{NNselfsim-lem} being particularly relevant.

\begin{remark} The notions `pseudo-idempotency' and `self-similarity' are precisely equivalent; the different terminology simply arises from different fields. Some authors (e.g. \cite{FL}) use the term `idempotent' for what we call `pseudo-idempotent', although \cite{JK} defines an idempotent to satisfy the much stricter condition $\code \otimes I_N \ = \ 1_N \otimes \code$.   
For consistency, we use the term {\em self-similar object} rather than {\em pseudo-idempotent}, unless usage is very well-established.
\end{remark}
Self-duality and self-similarity together are enough to characterise reflexive objects of compact closed categories.

\begin{lemma}\label{reflexchar-lem} Let $({\mathcal C},\_ \otimes \_ , \sigma_{\_ ,\_}, I, ( \ )^*)$ be a compact closed category. The reflexive objects of $\mathcal C$  are precisely the self-dual self-similar objects.
\end{lemma}
\begin{proof} \ Both parts of this proof are based on the very special form the internal hom takes in a compact closed category, as $[A\rightarrow B] \stackrel{def.}{=}A^*\otimes B$.   \\
$(\Rightarrow )$ Let $R\in Ob(\mathcal C)$ be a reflexive object, so there exist mutually inverse arrows 
$\app\in {\mathcal C}(R, [R\rightarrow R])$ and $\lam\in {\mathcal C}([R\rightarrow R], R)$.
As the dual is a contravariant monoidal functor, the following diagram commutes :
\[ 
\xymatrix{
R^*  & & [R\rightarrow R]^* \ar@{-}[d]^{1_{[R\rightarrow R]^*}} \ar[ll]_{\app^*} \\
R^* \ar@{-}[u]^{1_R^*}  \ar[rr]_{\lam^*} & & [R\rightarrow R]^* 
}
\]
Appealing to the identity $[R\rightarrow R] = R^*\otimes R$, we then derive isomorphisms exhibiting the self-duality of $R$.
\[ 
\xymatrix{
R  \ar[rr]^\app & & R^*\otimes  R \ar[rr]^{\sigma_{R^*,R}} & & R \otimes R^* \ar[rr]^{\app^*} & & R^*  \ar@{-}[d]^{1_{R^*}} \\
R  \ar@{-}[u]^{1_R}& & R^*\otimes  R \ar[ll]^{\lam}  & & R \otimes R^* \ar[ll]^{\sigma_{R,R^*}}  & & R^* \ar[ll]^{\lam^*} \\
}
\]
Once we have established that $R\cong R^*$ is self-dual, self-similarity or pseudo-idempotency of R follows as $ R \cong [R\rightarrow R] = R^*\otimes R \cong R\otimes R$. 
Thus reflexive objects are self-dual and self-similar.\\
$(\Leftarrow)$ This direction is even simpler. Let $R$ be both self-dual and self-similar. Then 
$R \ \cong \ R\otimes R \ \cong \ R^*\otimes R = [R\rightarrow R]$ and so $R$ is reflexive.
\end{proof}

\begin{remark}[Examples of self-similar and reflexive objects]
It is not difficult to give examples of either self-similar or reflexive objects.  From a general viewpoint, self-similarity dates back at least as far as Hilbert's parable of the Grand Hotel and Cantor's work on foundations of set theory (see \cite{NY13} for a good overview), and is significantly prefigured by Galileo's `infinity paradox'.  As observed in \cite{PHD} it is also a key feature of Girard's Geometry of Interaction program, and the term `self-similar' is used in a wide range of algebraic fields (e.g. \cite{ME}) motivated by a connection with fractal structures.  

Relevant concrete examples, along with Hilbert-hotel style bijections based on the natural numbers, are given in Lemma \ref{Nselfsim-sect} and Definitions \ref{Cantor-def}.

 For any self-similar object $X$ of a compact closed category, its tensor with its dual, $X\otimes X^*$  is not only self-similar, but self-dual.  Thus, provided a compact closed category does indeed contain self-similar objects, we have instant access to reflexive objects\footnote{It is perhaps easier to point out  compact closed categories that {\em do not} have reflexive objects. The compact closed category $({\bf Hilb_{fd}},\otimes)$ of finite-dimensional Hilbert spaces  with tensor product, as studied in the categorical quantum mechanics program, is a notable example. We refer to \cite{ABP} for the obstacles to defining compact closure in the infinite-dimensional (\& hence reflexive) case. It is tempting, although highly speculative, to relate this to the problems early pioneers of quantum computing had in attempting to define a `fully quantum' universal computer \cite{DD,JM}.}\label{hilbCC-foot}.
\end{remark}

Concrete examples of reflexive objects are given in Corollary \ref{NNreflex-corol}. However, at this point we cannot simply declare that we are done, and have established a suitable categorical setting for the type-freeness evident in \cite{GOI1,GOI2,AHS}. It is notable that all the action takes place within a single algebraic structure.  Instead, we treat reflexivity like many other categorical properties and consider how reflexivity may be `strictified' so that the isomorphisms exhibiting it become, under a suitable equivalence of categories, identity arrows.

\section{Strictification of reflexivity}
What we now need is a notion of `strictifying' reflexivity; moving -- via a categorical equivalence -- from a setting in which an object is reflexive is up to some  pair of isomorphisms, to a setting in which reflexivity may be exhibited by identity arrows.

\begin{definition}\label{strictref-def}
Given a reflexive object $R$ of a closed category $(\C,[\ \rightarrow \ ])$, we define a {\bf strictification} of the reflexivity of $R$ to be :
\begin{itemize}
\item a small closed category $\mathcal S$ and a faithful functor  of closed categories $\Gamma : \mathcal S \rightarrow \C$, 
\item a reflexive object $N\in Ob(\mathcal S)$ where $\Gamma(N)=R$,
\item a closed category $\mathcal D$, equivalent to $\mathcal S$,  in which this reflexivity is exhibited by identity arrows.
\end{itemize}
The above slightly convoluted description is required to avoid violating the principle of equivalence; in practice, we will simply refer to `a small closed subcategory of $\C$ that contains $R$', along with an equivalent closed subcategory where reflexivity is exhibited by identities. 

When $R$ is a reflexive object of a {\em monoidal} closed category $(\C , \otimes, [\_ \rightarrow \_], I )$we say that such a strictification is a {\bf monoidal strictification} when the relevant equivalence of categories is a monoidal equivalence, and the faithful functor $\Gamma : S\rightarrow \C$ is a monoidal functor. 
\end{definition} 

\begin{remark}\label{nonmon-rem}[Closure, and monoidal closure]
	Closure within the `non-monoidal closed categories' found in the somewhat obscure paper \cite{ML77} of M. Laplaza is undoubtedly the most `pure' form of closure available.  Additional categorical features associated with closure (such as the Cartesian product used in \cite{LS}) often correspond to additional logical or computational structures (in \cite{LS}, the product re-appears in logcial form, as `surjective pairing'), and a notion of closure defined without reference to any other operations beyond an internal hom. does not impose any additional notions on a logical or lambda-calculus interpretation.
	
	Despite this, we consider (monoidal) strictification of reflexivity within compact closed categories. Thus, we require an equivalence of categories preserves the unit object, elements of objects and names of arrows, and indeed the notion of monoidal well-pointedness.  
	
	This is partly for theoretical reasons; we are studying the setting of \cite{AHS,GOI1,GOI2}, rather than some more abstract notion of combinatory logic or lambda model. It is also partly for practical reasons; \cite{ML77} requires what have been referred to as `monstrous' coherence conditions, which would make any such strictification a decidedly non-trivial task -- we refer to \cite{MH14} for what is possibly the closest approach.
\end{remark}

We are therefore considering what is possibly the easiest case : monoidal strictification for a class of reflexive objects in compact closed categories, including those used in \cite{GOI1,GOI2,AHS}.  

\section{Strict duality, and strict self-similarity}
As the notion of reflexivity in compact closed categories simply splits into the notions of self-duality and self-similarity, we first consider what it means to have strict versions of both of these.

\subsection{Strict self-duality}
The notion of self-duality in the strict setting is well-studied :
\begin{definition}
A {\bf dagger} on a (monoidal) category $\C$ is defined in \cite{PS} to be a contravariant (monoidal)  functor $(\ )^\dagger :\C^{op}\rightarrow \C$ that is the identity on objects and satisfies 
$\left( (\ )^\dagger \right)^\dagger=Id_\C$.  Thus, a dagger can be thought of as a strict version of self-duality.  Simply as notation, we extend this definition slightly, and say that a dual $( \ )^*:\C^{op}\rightarrow \C$ is a {\bf dagger at an object $X\in Ob(\C)$} when $X$ is strictly self-dual. This is precisely equivalent to stating that the monoidal subcategory generated by $X$ is a dagger category.
\end{definition}

\begin{remark}
The best-known examples of daggers on compact closed categories are undoubtedly those found in the `categorical quantum mechanics' program \cite{AC}, where the dual in the compact closed category of finite-dimensional Hilbert spaces is a dagger.  However, we are not able to use this setting to provide reflexive objects (see footnote 5).    	
	
The notion of a dagger has occasionally been criticised as an `evil' concept (i.e. breaking the principle of equivalence by relying on a notion of equality between objects). Leaving aside the possibility of the purist leveling the same criticism at the identity functor, we take an alternative viewpoint : self-duality is simply a categorical property for which we may consider -- in the appropriate setting -- a strict form.
\end{remark}

The following is then a simple corollary of Lemma \ref{reflexchar-lem} above : 
\begin{corollary}\label{applamSS-corol}
Let $N$ be a reflexive strictly self-dual object of a compact closed category $(\C,\otimes,( \ )^*)$,  so $( \ )^*$ is a dagger at $N$. Then $N\in Ob(\C)$ is self-similar, and  the $lam$ and $app$ arrows exhibiting reflexivity 
\[ \lam  : [N\rightarrow N]\rightarrow  N \ \mbox{ and } \  \app   : N\rightarrow  [N\rightarrow N]  \]  
are also code / decode arrows exhibiting self-similarity
\[ \code : N \otimes N \rightarrow N \ \mbox{ and } \ \decode : N\otimes N \rightarrow N \]
\end{corollary}
\begin{proof}
The internal hom in a compact closed category is given by $[N \rightarrow N] \ = \ N^*\otimes N$. As the dual is a dagger at $N$, we derive $[N \rightarrow N] = N \otimes N$. The $lam$ and $app$ isomorphisms are then mutually inverse bijections exhibiting $N\cong N\otimes N$.  However, these are unique up to unique isomorphism \cite{JHRS}, so our result follows.
\end{proof}

\begin{corollary}\label{twoobjects-corol} Let $N$ be a reflexive strictly self-dual object of a compact closed category. Then 
\begin{enumerate}
\item  The monoidal subcategory generated by $N$ is compact closed,
\item All non-unit objects of this subcategory are isomorphic,
\item When $N$ is strictly reflexive, this subcategory has precisely two objects, one of which is the unit object.
\end{enumerate}
\end{corollary}
(Concrete examples of 1. and 2. are given in Part 4. of Proposition \ref{pNatresults-prop}, and of Part 3. in Definition \ref{G-def}). 

This convergence of the elements of reflexivity is particularly relevant for compact closed categories arising from the $\bf Int$ or $\bf GoI$ construction, where all self-dual objects are by construction isomorphic to some strictly self-dual object.

\subsection{Strict self-similarity}\label{strictselfsim-sect}
It is now inevitable that we consider how self-similarity may be strictified.  We are fortunate to be able to rely on a pre-existing coherence theorem and strictification procedure \cite{JHRS}.  However, somewhat inconveniently for our aim of finding a {\em monoidal} strictification of reflexivity,  this firmly and unavoidably lives within the theory of {\em semi-monoidal}  categories. 

\subsection{Monoidal and semi-monoidal categories} 
A semi-monoidal category is simply one that satisfies all MacLane's axioms for a monoidal category, except for those concerning the unit object.  The following definitions may be found in \cite{K}.
\begin{definition}\label{semi-mon-def}
A {\bf semi-monoidal category}  is a category $\C$ with a functor $\_\otimes \_:\C\times \C\rightarrow \C$ that is associative up to an object-indexed family of natural isomorphisms $\tau_{X,Y,Z}:X\otimes (Y\otimes Z)\rightarrow (X\otimes Y)\otimes Z$ satisfying MacLane's pentagon condition
\[  ( \  \tau_{W,X,Y}\otimes 1_Z)  \  \tau_{W,X\otimes Y,Z} \  (1_W \otimes  \  \tau_{X,Y,Z}) \ = \   \  \tau_{W\otimes X,Y,Z}  \  \tau_{W,X,Y\otimes Z} \] 
A functor between semi-monoidal categories that (strictly) preserves the tensor is a {\bf (strict) semi-monoidal functor}.  We assume the obvious definition of {\bf semi-monoidal equivalence of categories}.

A {\bf semi-monoidal monoid} is simply a small semi-monoidal category with precisely one object.   These lie on the border of algebra and category theory, and interesting algebraic structures frequently arise from core categorical ones in this setting (e.g. Section \ref{F-sect}).

\end{definition}

What is lost in the passage from monoidal to semi-monoidal categories is the following notions:

\begin{definition}
Let $(\C,\otimes ,I)$ be a monoidal category. The {\bf elements} of $X$ are members of the homset $\C(I,X)$, and the elements of the unit object $\C(I,I)$ are known as {\bf abstract scalars} \cite{SAAS}, by analogy with the case of vector spaces and linear maps.  

When the abelian monoid of abstract scalars is the singleton, we say that $I$ is a {\bf trivial unit}, or that $(\C,\otimes, I)$ has {\bf trivial scalars}.

When $\C$ is a monoidal closed category, the elements of $[X\rightarrow Y]$ are known as {\bf names} and are in 1:1 correspondence with the members of the homset $\C(X,Y)$. Finally, when $(\C,\otimes,I)$ is {\em compact closed}, there is also a 1:1 correspondence between $\C(X,Y)$ and the elements $\C([X\rightarrow Y],I)$; these are known as {\bf co-names}.  
\end{definition}

We treat monoidal categories as a special class of semi-monoidal categories.  When discussing potential unit objects in semi-monoidal categories, it is common to rely on A. Saavedra's characterisation\footnote{It is notable that Saavedra never explicitly stated that this precisely characterised units in the sense of MacLane-Kelly. This observation was made by J. Kock who also laid out the basics of the theory of semi-monoidal categories \cite{K}. He later extended this to a more general theory of `weak units' in collaboration with A. Joyal \cite{JK}.} of units \cite{AS}, as laid out in \cite{K,JK}. Thus `being a unit' is a property that an object may have, rather than a part of the definition.

\begin{definition}\label{Sunit-def}
A {\bf (Saavedra) unit} in a semi-monoidal category $(\C,\otimes , \tau)$ is an object $U\in Ob(\mathcal C)$ that is both  {\bf pseudo-idempotent} and {\bf cancellable} i.e. it is self-similar, and the functors $(U\otimes \_ ), (\_ \otimes U) : \C \rightarrow \C$ are fully faithful.
\end{definition}  

The theory of Saavedra units is particularly relevant for semi-monoidal monoids, where it allows us to characterise those semi-monoidal monoids whose unique object is a unit object. The following is taken from \cite{JHRS}.

\begin{proposition}\label{NSS-prop}
Let $(M,\star,\alpha)$ be a semi-monoidal monoid. Then $\star$ is strictly associative (i.e. $\alpha=1_M$) iff the unique object of $M$ is a unit object.
\end{proposition}

\subsection{An algebraic interlude}\label{F-sect}
Proposition \ref{NSS-prop} raises an obvious question : the canonical associativity isomorphisms for a (non-degenerate) semi-monoidal monoid  form a non-trivial group ---  is it the same in every case, and if so, which group this is?

One of the most familiar and well-studied objects in group theory is Richard Thompson's group $\mathcal F$. This was originally defined in terms of a representation,  as the group of homeomorphisms of the unit interval that are piece-wise linear and order-preserving, non-differentiable only at a finite number of dyadic rationals, and  have slope of the form $2^k$, $k\in \mathbb N$ on differentiable sections. 

We use the definition as a group presentation : 
\begin{definition}\label{F-def}
	{\bf Thompson's group $\mathcal F$} is defined by  
	\[ \mathcal F = \langle x_0,x_1,x_2,\ldots \ : \ x_i^{-1} x_j x_i = x_{j+1} \ \forall \ i<j \rangle \]
Note that this is not a {\em minimal} presentation;  it is well-established  that $\{ x_0,x_1\}$ generates the whole of $\mathcal F$. However, the required relators are significantly less intuitive.
\end{definition}

It is by now folklore that group of canonical associativity isomorphisms in a (non-unit) semi-monoidal monoid is precisely Thompson's $\mathcal F$.  This -- or at least statements equivalent to this claim -- have been presented and rediscovered many times, and the following list is not exhaustive!

\begin{remark}\label{FinHistory}[Connections between $\mathcal F$ and coherence for associativity] 
 As early as 1973, R. Thompson and J. McKenzie noted \cite{MKT} a connection with `{\em associativity laws}'. In \cite{MVL06}, M. Lawson considered a class of semi-monoidal monoids in the special case where the tensor $\_ \star \_$ admits projection / injection arrows (as studied in \cite{PHD,MVL,TAC}), and demonstrated that the group of canonical isomorphisms\footnote{A curious feature of \cite{MVL06} is that the link between the given representation of $\mathcal F$ and associativity isomorphisms is not made explicit by the author, despite having (along with other authors)  described and used the same operations as associativity isomorphisms in previous work \cite{MVL,PHD}.} is precisely $\mathcal F$. In \cite{PD96}, P. Dehornoy considered $\mathcal F$ abstractly, and noted that, `The only [non-trivial] relations in this presentation of $\mathcal F$ correspond to the well-known MacLane-Stasheff pentagon'. In \cite{FL}, M. Fiore \& T. Leinster considered the strict monoidal category freely generated by a generic (pseudo-)idempotent and proved that its symmetry group is precisely $\mathcal F$.   In \cite{MB}, M. Brin talks about, `the resemblance of the usual coherence
theorems with Thompson's group $\mathcal F$', and this observation was used in \cite{JHRS} to note that -- at least in the `free' case, canonical associativity isomorphisms for a semi-monoidal monoid are precisely a copy of $\mathcal F$.  M. Lawson recently updated his paper \cite{MVL06} in \cite{MVL20}, and gave a construction of $\mathcal F$ based on (finite) maximal binary prefix codes; any categorically-minded reader will identify this construction as functorial, and the relevant prefix codes as a representation of MacLane's monogenic category $W$ (excluding the unit object).
\end{remark}

The following theorem and outline proof is presented with no claim to originality; it is given simply because the monoidal category theory is often implicit, rather than explicit, in several of the references of Remark \ref{FinHistory} above. We also wish to connect the generators of the presentation given in Definition \ref{F-def} above with the category theory.
\begin{theorem}\label{Fassoc-thm}
Let $(M, \_ \star \_)$ be a (non-unit) semi-monoidal monoid. The canonical associativity isomorphisms for $\_ \star \_$ form a copy of Thompson's group $\mathcal F$.
\end{theorem}
\begin{proof}{\em [OUTLINE]}\\
Let us denote the canonical associator for $\_ \star \_$ by $\alpha \in M$, and define $\{ x_j \}_{j\in \mathbb N}$ inductively by $x_0=\alpha$, and $x_{i+1}=1\star x_i$.  Functoriality of $\_ \star \_$ and MacLane's pentagon then immediately give the defining relations of Thompson's $\mathcal F$, as  $x_i^{-1} x_j x_i = x_{j+1} \ \forall \ i<j $. Thus the group of canonical associativity isomorphisms contains a homomorphic image of $\mathcal F$. However, a standard fact (e.g. \cite{MB96}) about $\mathcal F$ is that it has no non-abelian quotients, and so this is precisely a copy of $\mathcal F$. Finally, we may appeal to MacLane's pentagon to demonstrate that all canonical associativity isomorphisms for $\_ \star\_$ are generated by the set $\{ x_j\}_{j\in \mathbb N}$.
\end{proof}

\begin{remark}
Thompson's group $\mathcal F$ is itself, of course, a semi-monoidal monoid. This observation was made -- in non-categorical terms -- by K. Brown in \cite{KB} where he describes a group homomorphism $\mu:\mathcal F \times \mathcal F \rightarrow \mathcal F$ that is ``associative up to conjugation by the generator $x_0$''.  It is also remarkable that -- despite \cite{KB} being phrased in entirely non-categorical terms -- K. Brown also proves that $\mu (1 , \_)$ and $\mu (\_ , 1)$ are injective, but cannot be surjective. In our terms, he is establishing precisely that J. Kock's conditions for a Saavedra unit are {\em not} satisfied by the unique object of $\mathcal F$. 
\end{remark}

\subsection{Functors between monoidal \& semi-monoidal categories}  
We return to category theory, and define several functors between the (large) categories of monoidal and semi-monoidal categories.
\begin{definition}\label{units_stuff-def}
We denote the (large) categories of monoidal categories and semi-monoidal categories by $\bf MonCat$  and $\bf SemiMonCat$ respectively.  There is an obvious faithful functor $\iota : {\bf MonCat} \rightarrow {\bf SemiMonCat}$ corresponding to the triviality that every monoidal category is also semi-monoidal.

Given a semi-monoidal category $(\C,\otimes)$, we may adjoint a strict unit object simply by 
taking the categorical coproduct with the terminal monoidal category, and extending the tensor by strictness. This process is functorial; we denote this functor by  $(\_ )_{+I}:{\bf SemiMonCat}\rightarrow {\bf MonCat}$. 
Note that by construction, $(\C,\otimes )_{+I}$ has trivial scalars, and no non-trivial elements.

Going in the other direction, given a monoidal category $(\mathcal M,\_\otimes\_ ,I)$, let us denote by $(\mathcal M,\_\otimes\_ )_{-I}$ the full subcategory consisting of all non-unit objects.  This `forgetting the unit' process is also functorial; we denote this by $(\_ )_{-I}:{\bf MonCat}\rightarrow {\bf SemiMonCat}$. 
The composite  $\left( (\ )_{+I}\right)_{-I} = Id_{\bf SemiMonCat}$.  This is, of course, a one-sided inverse;  $\left( ( \ )_{-I}\right)_{+I} : {\bf MonCat}\rightarrow {\bf MonCat}$ is certainly not the identity functor.  Rather, it has the effect of deleting elements (and therefore, when appropriate, names and co-names).  This endofunctor on $\bf MonCat$ will prove important; we refer to it as the {\bf de-element functor}, and denote it by   $\left( \ \right)_{-\eafw}  = \left( \left( \ \right)_{-I}\right)_{+I}  : \ {\bf MonCat}\rightarrow {\bf MonCat}$.
\end{definition}

\begin{remark}Given a compact closed category $(\mathcal C,\otimes)$, the de-element functor certainly does not result in a compact closed category; $(\mathcal C,\otimes)_{-\eafw}$ has lost elements, including abstract scalars,  names, co-names, \& the distinguished unit / co-unit maps.  However, it may still be closed in the sense of \cite{ML77}.    
\end{remark}  

\begin{remark}
A natural question is whether, given some abelian monoid $U$,  we may `extend' a monoidal category with trivial scalars to one where the abstract scalars are taken from this abelian monoid?   This is a surprisingly non-trivial task; a procedure for doing so in the case of traced and compact closed categories is given in \cite{SAAS}.  We discuss this further in Section \ref{fd-sect}
\end{remark}

\section{A coherence theorem for self-similarity}\label{SSstrict-sect}
We now describe the relevant points of the coherence theorem for self-similarity \cite{JHRS}, which lives firmly within the category $\bf SemiMon$.

The theory of semi-monoidal monoids is essentially interchangeable with the theory of self-similarity in semi-monoidal categories.  Let $(M,\star)$ be a monoid with a semi-monoidal tensor; the unique object $m\in Ob(M)$ of this monoid is clearly self-similar, as $m\star m = m$. Similarly, the endomorphism monoid of any strictly self-similar object is clearly a semi-monoidal monoid\footnote{A special case of this is strictly reflexive objects in compact closed categories.  However, this is due to the special form that the internal hom takes in a compact closed category; the same need not be true in arbitrary monoidal closed categories, and is certainly not the case for non-monoidal closed categories.}. Thus semi-monoidal monoids may be considered to be a `strict' form of self-similar objects. This observation was formalised in the coherence theorem and strictification procedure of \cite{JHRS}, where the following useful results may be found :
\begin{theorem}\label{ss-coh-thm} Let $N\in Ob(\C)$ be a self-similar object of a semi-monoidal category $(\mathcal C , \otimes)$, and let 
$\code \in \C(N\otimes N,N)$ and $\decode \in \C ( N,N \otimes N)$  
be the unique (up to unique isomorphism) {\em code} and {\em decode} bijections exhibiting this self-similarity. Then
\begin{enumerate}
\item\label{montensor-pt} The operation defined on the endomorphism monoid of $N$ by 
\[ f\star_{\dc} g = \code (f\otimes g)\decode  \ \ \forall f,g\in {\mathcal C}(N,N) \]
is a semi-monoidal tensor on ${\mathcal C}(N,N)$.
\item\label{equiv-pt} There is a semi-monoidal equivalence of categories between 
\begin{enumerate}
\item  the semi-monoidal subcategory of $(\mathcal C ,\otimes)$ generated by $N$, 
\item the semi-monoidal monoid $(\mathcal C(N,N),\_ \star \_)$.
\end{enumerate}
\item As a consequence of Proposition \ref{NSS-prop}, the equivalence of categories of Point \ref{equiv-pt} maps {\em strict} to {\em non-strict} associativity in the case where $\_ \otimes \_$ is strict but $N$ is not a unit object.
\end{enumerate}
\end{theorem}
\begin{proof}
We refer to \cite{JHRS} for proofs of the above. These proofs are based on giving necessary and sufficient conditions for the commutativity of diagrams over a certain class of primitives (the tensor $\otimes$ and its canonical isomorphisms, the tensor $\star$ and its canonical isomorphisms,  and the code / decode arrows).  
\end{proof}

\begin{remark}[On the choice of code / decode arrows] It is natural to wonder whether the choice of code / decode arrows is significant -- in particular, if we are intending to strictify {\em reflexivity}, should we not ensure that the distinguished $\app$ and $\lam$ of Corollary \ref{applamSS-corol} are chosen, as opposed to some other isomorphisms with the same source / target?
	
The answer to this lies in the observation of \cite{JHRS} that code / decode arrows are unique up to unique isomorphism. Thus, all strictifications of some self-similar object are semi-monoidally equivalent (algebraically, they are isomorphic semi-monoidal monoids), regardless of the choice of code / decode arrows. Changing the code / decode isomorphisms may be seen as moving to an isomorphic representation of the same structure -- this is studied in more detail in \cite{CTLLP}, where an analogy between this and changes of basis in matrix representations is formalised.  
\end{remark}

\begin{corollary}\label{Finall-corol} The endomorphism monoid of every self-similar object in a semi-monoidal category  contains a copy of Thompson's $\mathcal F$. (Note that this result holds even if the semi-monoidal category in question is strictly associative).
\end{corollary}

The `adjoining a strict unit' functor $(\ )_{+I}:{\bf SemiMonCat}\rightarrow {\bf MonCat}$ of Definition \ref{units_stuff-def} does not give us, for free, a monoidal, rather than semi-monoidal, equivalence. 
Rather, we have the following simple corollary:
\begin{corollary}\label{eafw-equiv-corol} Let $(\C,\otimes ,I)$ be a monoidal category, and denote by $(N^\otimes,\otimes ,I)$ the full monogenic monoidal subcategory generated by some self-similar object $N\cong N\otimes N$. Then there exists a monoidal equivalence of categories between $(\C , \star)_{+I}$ and $(N^\otimes,\otimes ,I)_{-\eafw}$.
\end{corollary}

\begin{remark}
As Corollary \ref{eafw-equiv-corol} emphasises, the strictification process for self-similarity of \cite{JHRS} naturally lives within $\bf SemiMonCat$ rather than $\bf MonCat$; Applying it in the monoidal setting gives an equivalence of `de-elemented' categories. Some work is needed in order to use it to strictify reflexivity in a compact closed category, even given a self-similar object with a dagger.

The key to doing this is the canonical construction of {\em compact closed categories} from {\em symmetric traced monoidal  categories}, given by either the {\bf Int} construction of \cite{JSV}, or the {\bf GoI} construction of \cite{SA96}.  Notably, this construction equips the compact closed category with a suitable range of elements (\& hence names and co-names), {even when the underlying traced category has none}.
\end{remark}

\section{From traced categories to compact closure}\label{IntGoI-sect}
It is well known that the two constructions of {\em compact closed categories} from {\em symmetric traced monoidal  categories} (the {\bf Int} construction of \cite{JSV} and the {\bf GoI} construction of \cite{SA96}) are equivalent in the symmetric case, although \cite{JSV} also considered the more general braided /  tortile monoidal categories. Despite this, they used significantly different conventions, and the equivalence between the two was given by E. Haghverdi in \cite{HA} (see also Footnote 8.).  

Somewhat perversely for a volume dedicated to the work of S. Abramsky, we will work with the conventions of \cite{JSV} instead\footnote{Although this has become convention, we do not simply act out of peer pressure.  Rather, we consider the conventions of \cite{JSV}  and \cite{SA96} to give two different, but isomorphic  compositions  (along with tensors, duals, etc.) on the same underlying structure. Our claim is that interesting category theory and algebra may arise out of taking a 2-category or bi-category approach, and considering the interaction of the two. 
This was (implicitly) the approach taken in the identification in \cite{PHD} of the cut / cut-elimination procedure in Girard's Geometry of Interaction as compact closure. This was written when the author was unaware of the work of \cite{SA96} and based on an early draft of \cite{JSV}. This  referred to a second, isomorphic, `vertical' composition on hom-sets of $\bf Int(\mathcal C)$, introduced in order to transform the `horizontal' composition of \cite{JSV} into something that directly matched the resolution formula and cut-elimination procedure within Girard's Geometry of Interaction system.  In retrospect, the (isomorphic) `horizontal' and `vertical' compositions are those derived from the respective $\bf Int$ and $\bf GoI$ constructions of \cite{JSV} and \cite{SA96}. See Remark \ref{IntGoIsquares-rem} for more details on this.}.

\subsection{Categorical traces}
We start with the definition of a traced symmetric monoidal category, as a special case of the more general braided monoidal categories of \cite{JSV}.
\begin{definition}\label{Trace-def}
A {\bf trace} on a symmetric monoidal category $(\mathcal C , \_ \otimes \_  , \sigma_{\_ , \_ }, I)$ is an object-indexed family of mappings of homsets $Tr^U_{X,Y}:\mathcal C(X\otimes U , Y \otimes U) \rightarrow \mathcal C(X,Y)$   
that is natural in $X$ and $Y$, dinatural in $U$, and satisfies the following axioms:
\begin{itemize}
\item {\bf (Vanishing I)} 
$Tr^I_{X,Y}(\ )= Id_{\C(X,Y)}$, for all $X,Y\in Ob(\C )$.
\item {\bf (Vanishing II)} $Tr^{U\otimes V}_{X,Y} = Tr^U_{X,Y}\left (Tr^V_{X\otimes U,Y\otimes U}(f)\right)$ for all $f:X\otimes U \otimes V \rightarrow Y \otimes U \otimes V$.
\item {\bf (Yanking)} $Tr^U_{U,U}(\sigma_{U,U})=1_U$.
\item {\bf (Superposing)} $Tr^U_{X,Y}(f)\otimes g = Tr^U_{X\otimes A,Y\otimes B}\left((1_Y\otimes \sigma_{B,U})(f\otimes g)(1_X\otimes \sigma_{A,U})\right)$ for all $f:X\otimes U \rightarrow Y\otimes U$, $g:A\rightarrow B$.
\end{itemize}
\begin{remark}\label{V1-rem} A consequence of the Vanishing I axiom  is that when $(\C,\otimes, I)$ has trivial scalars, traces are uniquely determined by their action on non-unit objects.  The Vanishing II axiom is also sometimes known as the `confluence axiom', for obvious reasons.\end{remark}
The diagrammatic calculus for traced \& compact closed categories is well-established in \cite{JS1,JS2,JSV}, where  traces appear as feedback loops :
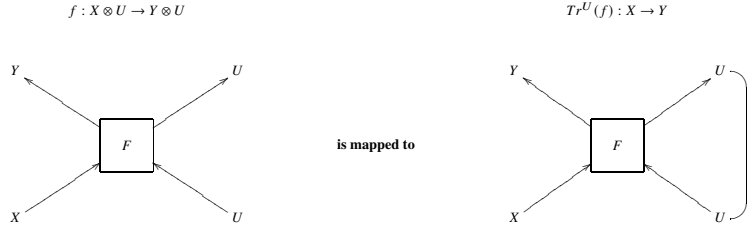
\begin{center}
\scalebox{0.6}{\small 
\xymatrix{
			&f:X\otimes U \rightarrow Y\otimes U		&				&												&			& 	Tr^U(f) : X\rightarrow Y			&						\\
Y			&								&	U			&												& Y			&								&	U	 \ar@{-} `r[d] `[dd] [dd] 	\\ 
			&*+<3em>[F]{ F} \ar[ur] \ar[ul]			&				&	\hspace{3em}\mbox{ \bf is mapped to }\hspace{3em}		&			&	*+<3em>[F]{ F} \ar[ur] \ar[ul]		&			\\ 
X\ar[ur]		&								&	U	\ar[ul]	&												& X \ar[ur]		&								&	U	\ar[ul]	\\ 	
}
}
\end{center}
\end{definition}

\section{From traces to compact closure}\label{Int}
We now give an exposition of the construction of compact closed categories from traced symmetric monoidal categories 
\noindent
\begin{definition}\label{Int-def}
Let $(\mathcal C , \otimes , \sigma , I, Tr_{\_,\_}^{\_}(\ ))$ be a traced symmetric monoidal category.  The compact closed category $({\bf Int \mathcal C},\_\Box\_ , (\ )^* , \epsilon_\_ , \eta_\_ )$ is defined as follows:
\begin{enumerate}
\item {\bf (Objects)} An object $(X,U)$ of $\bf Int \C$ is a pair of objects of $\mathcal C$.
\item {\bf (Arrows)} The homset ${\bf Int \C} ((X,U),(Y,V))$ is precisely $\C(X\otimes V, Y \otimes U)$.

\item {\bf (Composition)} Given arrows $f_\ud : (X,U)\rightarrow (Y,V)$ and $g_\ud : (Y,V)\rightarrow (Z,W)$, their composite $g_\ud \circ f_\ud\in  {\bf Int\C}((X,U),(Z,W))$ 
is defined using the trace, symmetry,  \& composition of the underlying category $\C$ (which we denote by concatenation), as follows :
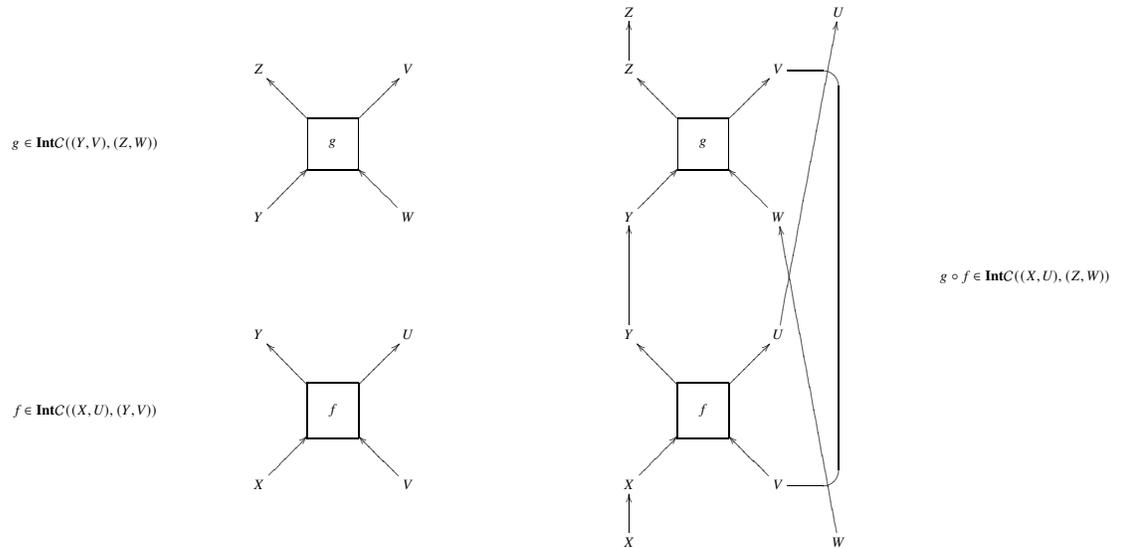
\begin{center}
\scalebox{0.6}{\small
\xymatrix{
						&&			&							&				& \hspace{8em} 	&		Z		&							&								&	U			&&	\\
						&&	Z 		&							&	V 			&				&		Z \ar[u]	&							&	V \ar@{-} `r[dr] `[dddddd] [dddddd]	&				&&	\\
g_\ud\in {\bf Int\C}((Y,V),(Z,W)) 	&&			&*+<3em>[F]{ g} \ar[ur] \ar[ul]		&				&				&				&*+<3em>[F]{ g} \ar[ur] \ar[ul]	&									&				&& \\
						&&	 Y \ar[ur]	&							&	W \ar[ul]		&				&		 Y \ar[ur]	&							&	W \ar[ul]						&				&&	\\
						&&			&							&				&				&				&							&								&				&&	g_\ud \circ f_\ud\in {\bf Int\C}((X,U),(Z,W))	\\			
						&&	Y  		&							&	U 			&				&		Y  \ar[uu]	&							&	U \ar[uuuuur]					&				&&	\\
f_\ud\in {\bf Int\C}((X,U),(Y,V)) 	&&			&*+<3em>[F]{ f} \ar[ur] \ar[ul]		&				&				&				&*+<3em>[F]{ f} \ar[ur] \ar[ul]		&								&				&&	\\
						&&	 X \ar[ur]	&							&	V \ar[ul]		&				&		 X \ar[ur]	&							&	V \ar[ul]						&				&&	\\
						&&	 		&							&				&				&		 X \ar[u]	&							&								&	W \ar[uuuuul]	&&	\\
} 
}
\end{center}
\item {\bf (Identities)} The identity at an object $1_{(X,U)}\in{\bf Int\C}((X,U),(X,U))$ is simply  $(1_X\otimes 1_U)\in \C(X\otimes U,X\otimes U)$. 
\item {\bf The Tensor} $\bf Int\C$  has a symmetric monoidal tensor, $\_ \Box \_$, given by:
\begin{itemize}
\item {\bf (Objects)} $(X_1,U_1) \Box (X_2,U_2) \ = \ (X_1\otimes X_2 , U_2 \otimes U_1)$ for all $(X_1,U_1),(X_2,U_2)\in Ob({\bf Int\C
})$.
\item {\bf (Arrows)}
Given arrows $f_\ud :(X_1,U_1)\rightarrow (Y_1,V_1)$ and $g_\ud :(X_2,U_2)\rightarrow (Y_2,V_2)$,
their tensor $f_\ud \Box g_\ud  : (X_1,U_1)\Box (X_2,U_2) \rightarrow (Y_1,V_1)\Box(Y_2,V_2 )$ 
is given diagrammatically, as
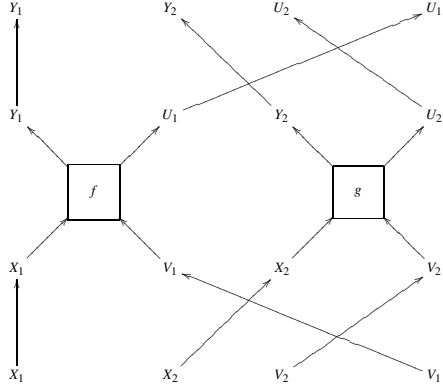
\begin{center}
\scalebox{0.6}{\small 
\xymatrix{
			&						&				&		&			&								&				\\
Y_1			&						&   Y_2			&		& U_2		&								&  U_1		\\
			&						&				&		&			&								&				\\
Y_1 \ar[uu]	&						&	U_1 \ar[uurrrr]	&		& Y_2 \ar[uull]	&								&	U_2	\ar[uull]	\\ 
			&*+<3em>[F]{ f} \ar[ur] \ar[ul]	&				&		&			&	*+<3em>[F]{ g} \ar[ur] \ar[ul]		&			\\ 
X_1\ar[ur]		&						&	V_1	\ar[ul]	&		& X_2 \ar[ur]	&								&	V_2	\ar[ul]	\\ 	
			&						&				&		&			&								& 					\\
X_1\ar[uu]		&						&	X_2 \ar[uurr]	&		& V_2\ar[uurr]	&								&	V_1\ar[uullll]			\\
			&						&				&		&			&								&				\\
}
}
\end{center}
\end{itemize}
\item {\bf (The unit object)} The unit object is simply $(I,I)$, where $I$ is the unit object of $\C$.
\item {\bf (The dual on objects)} This is defined by $(X,U)^* = (U,X)$. 
\item  {\bf (The dual on arrows)} This is defined in terms of the symmetry isomorphism of $\C$; given an arrow $f_\ud\in \C((X,U), (Y,V))$, then 
$\left( f_\ud \right)^* = \left( \sigma_{Y,U} f \sigma_{X,V}\right)_\ud \in {\bf Int\C}((U,X),(Y,V))$.
\item {\bf (The unit and co-unit)} The distinguished {\em unit} and {\em co-unit} arrows
$\eta:(I,I)\rightarrow (X,U) \Box(U,X)$ and $\epsilon : (U,X)\Box(X,U)\rightarrow (I,I)$
are specified by {\em symmetry arrows} in the underlying traced monoidal category $\mathcal C$, so  
\[ \eta_{(X,U)} =	\left( \sigma_{I,X\otimes U}\right)_\ud \in {\bf Int\C}((I,I),(X,U)\Box(U,X)) \]
and 
\[ \epsilon_{(U,X)}  = 	\left( \sigma_{U\otimes X,I}\right)_\ud \in {\bf Int\C}((U,X)\Box (X,U), (I,I)) \]
\end{enumerate}

\end{definition}

Although Joyal, Street, and Verity assumed strict associativity of the underlying traced monoidal category (and hence of the resulting compact closed category), and often left canonical isomorphisms implicit, it is nevertheless straightforward to write down the construction in the case where these canonical isomorphisms are made explicit. This rather thankless task was carried out in \cite{PHD} where the following may be found:

\begin{lemma}\label{Intcanonisos-lem}
Let $(\C , \otimes \sigma_{\_,\_} , \tau_{\_,\_,\_} , I , Tr(\ ) )$ be a traced symmetric monoidal category. Then the associativity and symmetry isomorphisms of the resulting compact closed category $({\bf Int\C}, \_ \Box \_)$ are given by 
\begin{description}
\item[Symmetry] For all $(X,U),(X',U')\in Ob({\bf Int\C})$,
\[ S_{(X,U),(X',U')} = \sigma_{X,X'}\otimes \sigma_{U,U'} 
\] 
\item[Associativity] For all $(X,U),(Y,V),(Z,W)\in Ob({\bf Int\C})$,
\[ T_{(X,U),(Y,V),(Z,W)} = \alpha_{X,Y,Z} \otimes \alpha^{-1}_{W,V,U} \]
\end{description}
\end{lemma}
\begin{proof} We refer to \cite{PHD} for the details. The key point is to demonstrate that the sources and targets are correct. 
\begin{description}
\item[Symmetry]
\[\C((X\otimes X')\otimes (U\otimes U'),(X'\otimes X)\otimes (U'\otimes U)) = {\bf Int\C}((X,U)\Box(X',U'),(X',U')\Box(X,U))    \]
\item[Associativity]
\[ \begin{array}{rl} & \C( \left( X\otimes (Y\otimes Z)\right) \otimes \left( (W \otimes V) \otimes U\right) , \left( (X\otimes Y)\otimes Z)\right) \otimes \left( W \otimes (V \otimes U)\right) \\ = & {\bf Int\C}((X,U)\Box((Y,V)\Box(Z,W)) , ((X,U)\Box(Y,V))\Box (Z,W)) \\ \end{array}
\]
\end{description}
\end{proof}


A common intuition is that $\bf Int\mathcal C$ is a `dualised' version of $\mathcal C $  that contains both $\mathcal C$ and $\mathcal C^{op}$.  This is apparent in the following, taken from \cite{JSV}.
\begin{proposition}\label{cocontra-prop}
There exist faithful traced monoidal functors $L_I ,R_I: \C \rightarrow {\bf Int\C}$ that are covariant and contravariant respectively. These are given by, for all $X,Y\in Ob(\mathcal C)$, and $f\in \mathcal C(X,Y)$,  
\begin{description}
\item[Covariant] $L_I(X)=(X,I)$ and $L_I(f)= (f\otimes 1_I)_\ud\in {\bf Int\C}((X,I),(Y,I))$
\item[Contravariant] $R_I(X)= (I,X)$ and $R_I(f) = (1_I\otimes f)_\ud\in {\bf Int\C}((I,X),(I,Y))$
\end{description}
\end{proposition}

Should we be prepared to consider semi-monoidal rather than monoidal functors, the above proposition generalises to arbitrary non-unit objects.
In this setting, we have no single distinguished covariant and contravariant faithful semi-monoidal functors from $(\mathcal C , \otimes)$ to $({\bf Int\C},\Box)$; rather, we have such functors indexed by the objects of $\mathcal C$.
\begin{proposition}\label{semi-cocontra-prop}
Given an arbitrary object $U \in Ob(\mathcal C)$ of some traced symmetric monoidal category, we may define both covariant and contravariant faithful 
semi-monoidal functors from $(\mathcal C,\otimes)$ to $(\bf Int \mathcal C, \Box)$ by, for all $X,Y\in Ob(\mathcal C)$, and $f\in \mathcal C(X,Y)$,  
\begin{description}
\item[Covariant] $L_U(X) = (X,U)$ and $L_U(f) = (f\otimes 1_U)_\ud\in {\bf \C}((X,U),(Y,U))$
\item[Contravariant] $R_U(X) = (U,X)$ and $R_U(f)  =  (1_U\otimes f)_\ud\in {\bf Int\C}((U,X),(U,Y))$
\end{description}
\end{proposition}
\begin{proof} This follows in precisely the same way as the standard proofs of Proposition \ref{cocontra-prop}; we are simply not insisting that our functors preserve a unit object.
\end{proof}

\begin{corollary}\label{selfdualdagger-corol}
As a well-established corollary, objects of the form$(N,N)\in Ob({\bf Int\C})$ are self-dual, for arbitrary $N\in Ob(\mathcal C)$, and all self-dual objects are isomorphic to some object of this form. 
\end{corollary}

\begin{corollary}\label{intselfsim-corol}
A self-dual object $(N,N)\in Ob({\bf Int\C})$ is self-similar iff $N\in Ob(\C)$ is self-similar.
\end{corollary}
\begin{proof}\ \\ 
$(\Leftarrow)$
 This is immediate, and well-established (e.g. \cite{PHD}). Given code / decode arrows $\code\in \mathcal C( N\otimes N , N)$ and $\decode \in \mathcal C(N ,N\otimes N)$, the self-similarity of $(N,N)\in Ob({\bf Int\mathcal C})$ is exhibited by 
\[ (\code \otimes \decode) \in {\bf Int\C}((N,N)\Box(N,N),(N,N)) \ \ \mbox{ and } \ \ (\decode \otimes \code) \in {\bf Int\C}( (N,N),(N,N)\Box(N,N)) \]
$(\Rightarrow)$  Consider some $(N,N)\cong (N,N)\Box(N,N)$. then 
\[ (N,N)\ \cong \ (N,I)\Box(I,N) \ \cong \ \left( (N,I)\Box(I,N)\right) \Box \left((N,I)\Box(I,N) \right) \]
As the tensor of $\bf Int\C$ is symmetric, this implies 
\[  (N,I)\Box(I,N)\ \cong \  \left( (N,I)\Box(N,I)\right) \Box \left((I,N)\Box(I,N) \right) \ \cong \ (N\otimes N,I)\Box(I,N\otimes N) \]
and our result follows as the functors $L_I,R_I: \C \rightarrow {\bf Int\C}$ are faithful.
\end{proof}

As a corollary of the above two results, we observe that is is straightforward, at least in principle,  to exhibit strictly reflexive objects of compact closed categories derived from the $\bf Int$ construction.
\begin{corollary}\label{ssIntss-corol}
Let $N\in Ob(\C)$ be a strictly self-similar object of a symmetric traced monoidal category. Then $(N,N)\in Ob({\bf Int\C})$ is both strictly self-dual and strictly self-similar, and hence strictly reflexive.
\end{corollary}
\begin{proof} Strict self-duality is immediate.  The self-similarity of $(N,N)$ is, by Corollary \ref{intselfsim-corol}, exhibited by 
\[ (\code \otimes \decode) \in {\bf Int\C}((N,N)\Box(N,N),(N,N)) \ \ \mbox{ and } \ \ (\decode \otimes \code) \in {\bf Int\C}( (N,N),(NN)\Box(N,N)) \]
However, when $\code = 1_N=1_{N\otimes N}=\decode$ these are both identity maps, so $(N,N)$ is strictly self-similar.
\end{proof}
The above result is not as helpful as it first appears;  it is not immediate how to find strictly self-similar objects of traced monoidal categories; we do not have a consistent method of strictifying self-similarity in  monoidal (rather than semi-monoidal) setting, so some work remains before we can achieve our goal of strictifying reflexivity in compact closed categories.

\subsection{A naming of parts}\label{names-sect}
A useful aspect of the $\bf Int$ construction is that it simply creates names (\& indeed elements generally) {\em ex nihilo}.  Less dramatically, we observe that even when the underlying traced monoidal category $(\C ,\otimes)$ has no elements, the category $\bf Int\C$ is, by construction, fully equipped with names for all its arrows (\& hence a wide range of elements).

\begin{lemma}\label{corresponding-lem}Let $(\mathcal C ,\otimes ,\sigma , I , Tr_{\_,\_}^{\_} )$ be a traced symmetric monoidal category. Then :  
\begin{enumerate}
\item For arbitrary  $(U,V)\in Ob({\bf Int\C})$,  the elements of $(U,V)$ are in 1:1 correspondence with the homset $\C(V,U)$.
\item For arbitrary $X\in Ob(\C)$ , the following are in 1:1 correspondence:
\begin{itemize}
\item elements of $X\in Ob(\C)$,
\item elements of $L_I(X)\in Ob({\bf Int\C})$, 
\item elements of $R_I(X)\in Ob({\bf Int\C})$.
\end{itemize}
\item For an arbitrary self-dual object $(N,N)\in Ob({\bf Int\C})$, the elements of $(N,N)$ are in 1:1 correspondence with the endomorphism monoid $\C(N,N)$.
\end{enumerate}
\end{lemma}
\begin{proof}\ \\
\begin{enumerate}
\item From the definition of homsets, ${\bf Int\C}((I,I),(U,V))  \stackrel{def.}{=} \C (I\otimes V,U\otimes I) \cong \C(V,U)$.
\item As a special case of 1., ${\bf Int\C}((I,I),L_I(X)) \stackrel{def.}{=} \C (I\otimes I,X\otimes I) \cong \C(I,X)$ and ${\bf Int\C}((I,I),R_I(X))  \stackrel{def.}{=}  \C (I\otimes I,I\otimes X)  \cong \C(I,X)$.
\item This is again a special case of 1.
\end{enumerate}
\end{proof}

\section{Strictifying reflexivity in a compact closed category}\label{strictify-sect}
The observations of Lemma \ref{corresponding-lem}, although straightforward, provide a route to the monoidal equivalences of compact closed categories we need in order to give a monoidal strictification of reflexivity.  The following preliminary results are needed:
\begin{proposition}\label{keyequivalences-prop}
Let $(\C,\otimes)$ be a symmetric traced monoidal category with trivial scalars. Then 
\begin{enumerate}
\item The de-elemented version $(\C,\otimes)_{-\eafw}$ is also traced.
\item For arbitrary non-unit $N\cong N \otimes N\in Ob(\C)$ there is a $\dagger$ - isomorphism\footnote{The terminology  `isomorphic' rather than `equivalent' is deliberate; these are both small categories, and there is a bijection of sets of objects as well as homsets.} of $\dagger$-compact closed categories between 
\begin{enumerate}
\item The monoidal subcategory of $\bf Int\C$ generated by $(N,N)$.
\item The  monoidal subcategory of $\bf Int\left(\C_{-\eafw}\right)$ generated by $(N,N)$. 
\end{enumerate}
\item When $N$ is self-similar, the  monoidal subcategory of ${\bf Int}((\C(N,N),\star)_{+I})$ generated by $(N,N)$ is $\dagger$ monoidal equivalent to a. and b. above.
\end{enumerate}
\end{proposition}
\begin{proof} \ \\
\begin{enumerate}
\item This is a straightforward consequence of the Vanishing 1 axiom; as observed in  Remark \ref{V1-rem}, tracing out the unit object is the identity on homsets. Conversely, as there is only one abstract scalar of $(\C,\otimes)_{-\eafw}$, traces of the form $Tr_{I,I}^U(\  )$ are also uniquely determined.
\item Parts a. and b. follows from   Part 3. of Lemma \ref{corresponding-lem} above; the only homsets of either of these compact closed categories determined by elements of the underlying traced category  are the abstract scalars, which are trivial in both cases.  It makes no difference whether or not we start with the `de-elemented' version.
\item This follows from the monoidal equivalence of categories noted in Corollary \ref{eafw-equiv-corol} (and indeed the fact that the $\bf Int$ construction is not `evil'; given monoidally equivalent traced monoidal categories $\C$ and $\mathcal D$, there exists a monoidal equivalence of categories between $\bf Int(\C)$ and $\bf Int(\mathcal D)$).  

We may also exhibit the trace  on $(\C(N,N),\star)_{+I}$ explicitly; in \cite{PHD} it is observed that for a self-similar object $N$ of a symmetric traced monoidal category, there exists an operation $trace$ on $(\C(N,N),\star)$ given by 
\[ trace(f)\ =\  Tr_{N,N}^N(\decode f \code) \ \ \forall f\in \C(N,N) \]
that `satisfies all of the axioms of \cite{JSV} apart from Vanishing I'.   As $\C$ has trivial scalars, this extends uniquely (as described in Remark \ref{V1-rem}) to a categorical trace on the symmetric monoidal category $(\C(N,N),\star)_{+I}$, and coincides, up to monoidal equivalence, with the trace on the subcategory of $\C$ monoidally generated by $N$. 
\end{enumerate}
\end{proof}

This now gives the desired monoidal strictification of extensional reflexivity. 
\begin{theorem}\label{strictproc-thm}
Let $\left(\C,\otimes ,\sigma, I, Tr(\ )\right)$ be a symmetric traced monoidal category with trivial scalars, and let $X\in Ob({\bf Int\C})$ be a reflexive object. Then there exists a monoidal strictification of the reflexivity of $X$.
\end{theorem}
\begin{proof}
By Lemma \ref{reflexchar-lem}, $X$ is self-dual, and from Corollary \ref{selfdualdagger-corol} it is therefore isomorphic to $(N,N)\in Ob(\bf Int\C)$ for some $N\in Ob(\C)$. Without loss of generality, we therefore work with the strictly self-dual reflexive object $(N,N)\in Ob({\bf Int\C})$.   As $X$ is self-similar, $N$ is therefore a self-similar object of $(\C ,\otimes)$, by Corollary \ref{intselfsim-corol}. 

By Proposition \ref{keyequivalences-prop} above, the monoidal subcategory of ${\bf Int}((\C(N,N),\star)_{+I})$ generated by $(N,N)$ then gives a monoidal strictification of this reflexivity, as axiomatised  in Definition \ref{strictref-def}.
\end{proof}

\subsection{Discussion}\label{strictdisc-sect}
We have exhibited a process that will strictify the reflexivity of a large class -- although not all -- reflexive objects in compact closed categories. 
There are several notable points.
\begin{enumerate}
\item 
The two restrictions on this process are : 
\begin{enumerate}
\item The compact closed category itself needs to arise from applying the $\bf Int$ or $\bf GoI$ construction, to a traced monoidal category.
\item The underlying traced monoidal category (and hence the compact closed category itself) must have trivial scalars. 
\end{enumerate} 
The first condition rules out compact closed categories such as finite-dimensional Hilbert spaces with tensor product (which does not, in any case, have reflexive objects), or  relations with Cartesian product (which certainly does have reflexive objects).  However, it is satisfied by the categories on which the original Geometry of Interaction system was based.

The second condition is also satisfied by the setting of the original Geometry of Interaction system that we will analyse in Section \ref{concrete-sect} onwards. It is harder to think of `natural' examples that are ruled out by this restriction, although  we may certainly construct examples by using the techniques of \cite{SAAS} to `adjoin' a non-trivial involutive monoid of abstract scalars. This is discussed further in Section \ref{fd-sect}.
\item The object $(N,N)$ of the compact closed category ${\bf Int}\left(  (\C(N,N),\star )_{+I} \right)$ is strictly self-similar, and therefore its endomorphism monoid is a semi-monoidal monoid. At this point, it is tempting to axiomatise away all the problems of Section \ref{ccmprobs-sect}, and simply {\em define} a `compact closed monoid' to be the result of such a process.   This would be inaccurate, and on a par with simply defining a compact closed category to be some subcategory of one that arises from (a restricted version of) the $\bf Int$ or $\bf GoI$ construction. 
\item\label{samehomsets-pt} A curiosity of this setting is that the endomorphism monoid of $N$ in the traced monoidal category $(\C(N,N),\star )_{+I}$ has the same underlying set as the endomorphism monoid of $(N,N)$ in the compact closed category ${\bf Int}\left(  \C(N,N),\star )_{+I} \right)$, since $N$ is {\em strictly} self-similar.  This is a very useful property when we come to study  such endomorphism monoids from a more algebraic setting.   
\end{enumerate}

\begin{conjecture}
The restrictions of Point 1. above are not essential, and reflexivity may be strictified for arbitrary reflexive objects in compact closed categories.\end{conjecture}

\section{Properties of strictly reflexive objects}
We now move on from demonstrating that -- at least in certain cases -- reflexivity may be strictified, and consider the structure of strictly reflexive objects in compact closed categories more generally.  Our setting includes, but is not restricted to, objects derived from the strictification of reflexivity given in Theorem \ref{strictproc-thm}.

\subsection{Notation \& an algebraic perspective} 
As observed in Point \ref{samehomsets-pt} of Section \ref{strictdisc-sect}, for a strictly self-similar object $N$ of a traced symmetric monoidal category $\C$, the endomorphism monoids $\C(N,N)$ and ${\bf Int\C}((N,N),(N,N))$have the same underlying set.  The $\bf Int$ construction provides a new composition, tensor, etc. on this set. With suitable notational discipline, we may treat the `old' and `new' compositions \& tensors as operations on the same set, and even consider their interaction. 

This is not done simply in order to horrify categorically-minded readers (although it will undoubtedly have that effect), but neither is it just a notational trick. 

The historical precedent and motivation for this was not originally phrased categorically.  In Girard's system, it is notable that the dynamics (i.e. cut and cut-elimination) is modeled by the composition of an endomorphism monoid of a {\em compact closed} category, whereas the model of conjunction is the tensor $\_ \star \_$ of the underlying {\em traced} category. The models of the exponentials are different again, and the bang $!(\ )$ is best seen as a right fixed-point semi-monoidal endofunctor $f\star !(f) = !(f)$ for the tensor modeling the conjunction -- but is derived from yet another monoidal tensor that distributes over the tensor modeling conjunction. 

The system as a whole relies on treating all these as operations on the same underlying set, and indeed considering their interaction.  Outside of models of untyped systems, it is hard to account for all this; within the untyped (or strictly reflexive) setting it is more algebraically natural, but still category-theoretically disturbing.
\\

Such notational abuse allows us to write down the following

\begin{proposition}\label{notationabuse-prop}
Let $(\C ,\otimes , s_{\_ , \_} , a_{\_,\_,\_} , Tr_{\_ , \_}^{- })$ be a symmetric traced monoidal category, and let $N=N\otimes N\in Ob(\C)$ be a (non-unit) strictly self-similar object. 
$(\C(N,N),\otimes)$ is then a semi-monoidal monoid;  for simplicity, let us denote its underlying set by $M$, its composition by $\cdot$,  its tensor by $\otimes$, and its canonical associativity isomorphism \& inverse by  
\[  \tau=a_{N,N,N} \in \C(N\otimes (N\otimes N), (N\otimes N)\otimes N)=\C(N,N) \] and 
\[  \tau'=a^{-1}_{N,N,N} \in \C((N\otimes N)\otimes N, N\otimes (N\otimes N))=\C(N,N)  \]
and its canonical symmetry isomorphism by 
\[ \sigma=s_{N,N}\in \C(N\otimes N , N \otimes N)=\C(N,N) \]

In the compact closed category $({\bf Int\C},\Box)$, the object $(N,N)$ is strictly reflexive,  and its endomorphism monoid is a semi-monoidal monoid with underlying set $M$. Let us denote its composition by $\circ$, its tensor by $\Box$,  its canonical associativity isomorphism \& inverse by $T,T'\in M$ and its symmetry isomorphism by $S\in M$. 

The following are then immediate :  
\begin{enumerate}
\item $(M,\circ)$ and $(M,\cdot)$ share the same identity element, $1_M$.
\item $T=\tau\otimes \tau^{-1}$ 
\item $S=\sigma \otimes \sigma$
\item The functions $(1\otimes \_) , (\_ \otimes 1) : (M,\cdot) \rightarrow (M,\cdot)$ are homomorphic self-embeddings of the monoid $(M,\cdot)$.
\item The functions $(1\Box \_) , (\_ \Box 1) : (M,\circ) \rightarrow (M,\circ)$ are homomorphic self-embeddings of the monoid $(M,\circ)$.
\item The functions $(1\otimes \_) , (\_ \otimes 1) : (M,\cdot) \rightarrow (M,\circ)$ are homomorphic  and anti-homomorphic  monoid embeddings respectively.
\item The subset $\{ \tau , \tau' \}\subseteq M$ generates, by closure under the composition $   \cdot $ and the tensor $ \otimes$, a subgroup of $(M,\cdot)$ isomorphic to Thompson's group $\mathcal F$.
\item The subset $\{ T , T' \}=\{ \tau\otimes \tau' , \tau'\otimes \tau\} \subseteq M$ generates, by closure under the composition $\circ$ and the tensor $\Box$, a subgroup of $(M,\circ)$ isomorphic to Thompson's group $\mathcal F$.
 \end{enumerate}
\end{proposition}
\begin{proof}
Part 1. is immediate from the definition of the $\bf Int$ construction, and parts 2. - 3. are similarly immediate from Lemma \ref{Intcanonisos-lem}.  Parts 4. and 5.  are general results on semi-monoidal monoids found in \cite{JHRS}.  Part 6. is derived from Proposition \ref{semi-cocontra-prop}. Finally, parts 7. and 8. are immediate from Theorem \ref{Fassoc-thm}.
\end{proof}

We now move on to horrify further  the categorically-minded reader, and consider some interactions between the canonical isomorphisms for $\star$, and the tensor  $\Box$ and composition $\circ$. Doing so leads us to the (rather categorically respectable) theory of Frobenius algebras, as well as monoids derived from Frobenius algebras that have close connections with classic structures from both group and semigroup theory.  

\begin{remark} As indicated in Footnote 9., we may consider the composition, tensor, dual, etc. derived from Abramsky's {\bf GoI} construction, rather than Joyal, Street, \& Verity's $\bf Int$ construction, as yet another collection of operations on the same underlying set -- together with a non-trivial bijection of the underlying set that will map between the two. 

In the following sections, we consider the interactions between the composition, tensor, and canonical isomorphisms of the underlying traced monoidal category with those of the derived from the $\bf Int$ construction.  We could of course do the same with the operations derived from the $\bf GoI$ construction instead, and derive an entirely distinct (\& similarly interesting) set of non-trivial interactions.  

This brings us to another reason for using the conventions of \cite{JSV} instead of those of \cite{SA96} -- we derive a more direct route to an interesting class of Frobenius algebras \& monoids introduced by Abramsky \& Heunen in \cite{AH}.
\end{remark} 

\subsection{Reflexivity and (unitless) Frobenius algebras \& monoids}\label{Fmonoid-sect}
As well as the fundamental r\^ole of compact closure in Abramsky \& Coecke's categorical quantum mechanics program \cite{AC},  another key building block is the (closely connected -- see Proposition \ref{SDFA-thm}) notion of a {\em Frobenius algebra}.  Beyond their well-known applications in quantum field theory \cite{JK}, in quantum information they model fan-out operations \cite{HS} in quantum circuits, and play the r\^ole of orthonormal bases in purely categorical formulations \cite{CP,CPV} of quantum foundations.  

In stark contrast to compact closure itself, the notion of a `unitless Frobenius algebra' is well-established, both in terms of theory and examples. These were defined and studied by S. Abramsky and C. Heunen in \cite{AH}, from where the following definition (but not terminology) is taken :   
\begin{definition}\label{FA}
Let $(\mathcal C ,\otimes , \alpha_{\_,\_,\_})$ be a  semi-monoidal category. 
An {\bf Abramsky-Heunen} or {\bf A-H Frobenius algebra} $\mathcal F = (S, \Delta, \nabla)$ in $\mathcal C$ consists of an object $S\in Ob(\C)$, equipped with an associative {\bf split arrow} $\Delta :S\rightarrow S \otimes S$ and a co-associative {\bf merge arrow} $\nabla:S\otimes S \rightarrow S$. These are required to satisfy the {\bf Frobenius condition} 
\[ (1_S\otimes \nabla) \alpha_{S,S,S}^{-1}(\Delta\otimes 1_S) = \Delta \nabla = (\nabla \otimes 1_S)\alpha_{S,S,S}(1_S\otimes \Delta) \]
Explicitly, {\bf associativity} and {\bf co-associativity} are the requirements that 
\begin{itemize}
\item $\nabla(1_S\otimes \nabla) = \nabla(\nabla\otimes 1_S)\alpha_{S,S,S}$  
\item $(1_S \otimes \Delta)\Delta=\alpha_{S,S,S} ((\Delta \otimes 1_S)\Delta) $. 
\end{itemize}
\end{definition}
Every Frobenius algebra is -- by  definition -- an A-H Frobenius algebra, but the converse is not true. In particular, as the above definition does not mention the unit object, there are therefore no obstacles to considering A-H Frobenius algebras whose distinguished object is the unique (non-unit) object of a semi-monoidal monoid.

\begin{definition}\label{AHFM-def}
Let $(M,\star ,\alpha)$ be a semi-monoidal monoid, and let $\Delta$ and $\nabla$ be the split and merge arrows of an A-H Frobenius algebra at the unique object of this monoid.  We define its {\bf A-H Frobenius} or {\bf A-H F monoid} to be the semi-monoidal submonoid of  $(M,\star)$ generated by the closure of $\{ \alpha , \alpha^{-1} , \Delta , \nabla \}$ under composition and tensor.
\end{definition}

\begin{remark}[Another variation of terminology] It is more standard to refer to the `split' and `merge' arrows of a Frobenius algebra as the {\em co-monoid} and {\em monoid} arrows respectively.  We do not do this, in order to avoid the potentially fatal confusion of terminology that would result.
\end{remark}

An interesting class of examples of Frobenius algebras is given by the following well-known result :
\begin{theorem}\label{SDFA-thm}
For every object $X$ of a compact closed category $(\C ,\_ \otimes\_ , \sigma_{\_,\_} , I , \epsilon\_ , \eta\_)$, the object $X^*\otimes X$ is self-dual and there is a Frobenius  algebra (the {\bf canonical Frobenius algebra}) at $X\otimes X^*$ with distinguished arrows $\Delta$ and $\nabla$ given by the following composites:
\[ \xymatrix{
X^* \otimes X \ar@{-}[r]^\cong & X^*\otimes I \otimes X \ar[rr]^{1_X^*\otimes \eta_X \otimes 1_X} & & X \otimes X^* \otimes X^*\otimes X } \]
and 
\[ \xymatrix{
X^* \otimes X \otimes X^*\otimes X  \ar[rr]^{1_{X^*}\otimes \epsilon_X\otimes 1_X} & &  X^*\otimes I \otimes X \ar@{-}[r]^{\cong } & X^* \otimes X
} 
\]
\end{theorem}
\begin{proof}
We refer to, for example, \cite{JV} for a direct proof of this, but note that it follows abstractly from a more general result of \cite{AL}, where it is proved for any object with an ambidextrous adjunction -- with the self-dual compact closed case following as a special case.
\end{proof} 

\begin{corollary}\label{onesidedinverse-corol} In Theorem \ref{SDFA-thm} above,  $\nabla\Delta\cong 1_{X^*\otimes X}\otimes \mu$ for some abstract scalar $\mu\in \C(I,I)$, and hence when $(\C,\otimes,I)$ has trivial scalars $\nabla\Delta=1_{X^*\otimes X}$.  However, this is a one-sided, rather than two-sided inverse. 
\end{corollary}
\begin{proof} The composite  $\nabla\Delta\cong 1_{X^*\otimes X}$ differs from the identity on $X^*\otimes X$ by the composite of a unit and a co-unit (i.e. a closed loop) in a compact closed category. Therefore, by \cite{SAAS}, it is the identity up to an abstract scalar, and is precisely the identity when the monoid of scalars is trivial.  To see why it is not a two-sided inverse, note that were this to be the case, every self-dual object of any compact closed category would be self-similar \& hence reflexive.
\end{proof}

\begin{definition}\label{AHFMreflex-def} Let $N\cong X \otimes X^*$ be a self-dual object of a compact closed category $\C$ . We refer to the A-H Frobenius algebra at $N$ given by Theorem \ref{SDFA-thm} above as the {\bf standard A-H Frobenius algebra}  at $N$. 

In the case where $N$ is a strictly reflexive object of $\C$, we observe that the endomorphism monoid of $N$ is a semi-monoidal monoid, and therefore (following Definition \ref{AHFM-def}) contains an A-H Frobenius monoid. We refer to this as the {\bf standard A-H F. monoid} at $N$.  When $\C$ also has trivial scalars, we refer to this as the {\bf simple A-H F monoid}.
\end{definition}

\begin{conjecture}
We conjecture that ``non-degenerate standard A-H Frobenius monoids differ only by their scalars''. More precisely, but less generally, all (non-abelian) simple A-H F. monoids are isomorphic.
\end{conjecture}

The simple A-H F monoid is worth studying as much for its algebra as its category theory. The following places it in the mainstream of  semigroup theory and group theory, as we demonstrate in Section \ref{algebra-sect}.  

\begin{lemma} Let $N$ be a strictly reflexive object of a compact closed category $(\C,\otimes, t_{\_,\_,\_})$ with trivial units. Then the split \& merge maps $\Delta ,\nabla \in \C(N,N)$ of the simple A-H F monoid satisfy:
\begin{enumerate}
\item $\nabla \Delta =1_N$
\item $\Delta\nabla=1_N$ iff $N$ is the unit object.
\end{enumerate}
\end{lemma}
\begin{proof}
Part 1. is immediate from Corollary \ref{onesidedinverse-corol} above. For part 2., let denote $t_{N,N,N}\in \C(N,N)$ by $\alpha$, and assume that the inverse of part 1. is a two-sided inverse. In this case, the  associativity axiom, $\nabla(1\otimes \nabla) = \nabla(\nabla\otimes 1)\alpha$ implies $(1\otimes \nabla)=(\nabla\otimes 1)\alpha$, giving $\alpha=\Delta \otimes \nabla $. Similarly, the co-associativity axiom implies   $\alpha=\nabla \otimes \Delta$ and so $\alpha=\alpha^{-1}$. Therefore, associativity at this semi-monoidal monoid is strict. By Proposition \ref{NSS-prop} this is the case iff $N$ is the unit object.
\end{proof}

In certain cases (including those arising from the strictification of reflexivity procedure described in Theorem \ref{strictproc-thm}), we are able to give an explicit description of the standard A-H F monoid at a strictly reflexive object.  The key to this is that in the case where the compact closed category of Theorem \ref{SDFA-thm} above arises from the $\bf Int$ construction, the {\em split} and {\em merge} arrows have a particularly neat form : 
\begin{proposition}\label{canonicalsplitmerge-prop}
Let $(C,\otimes , \alpha_{\_,\_,\_},\sigma_{\_,\_})$ be a traced symmetric monoidal category. Then the defining {\em split} and {\em merge} arrows of the standard Frobenius algebra at a self-dual  object $(N,N)\cong (N,I)\Box(I,N)\in Ob({\bf Int\C})$,
\[ \Delta \in {\bf Int\C}( (N,N),  (N,N)\Box(N,N)) \ = \  \C(N\otimes (N\otimes N),(N\otimes N)\otimes N) \]
and 
\[ \nabla \in {\bf Int\C}( (N,N)\Box(N,N), (N,N)) \ = \  C((N\otimes N)\otimes N,N\otimes (N\otimes N)) \]
are given by the symmetry and associativity isomorphisms of the underlying traced category, as  
\begin{itemize}
\item $\Delta \ =  \ \alpha_{N,N,N} (1\otimes \sigma_{N,N})$
\item $\nabla \ = \ (1\otimes \sigma_{N,N})\alpha_{N,N,N}^{-1}$
\end{itemize}
respectively.
\end{proposition}
\begin{proof}
This is immediate, and simply follows from substituting in the definitions of the unit / co-unit maps of $\bf Int\C$ into the definitions of Theorem \ref{SDFA-thm}.
\end{proof}

\begin{remark} The above observation, although simple, breaks the connection between the standard Frobenius algebra at a self-dual object of $\bf Int\C$ and the distinguished unit / co-unit maps -- and hence, the unit object.  Instead, the standard Frobenius algebra is simply derived from the canonical isomorphisms of the underlying traced symmetric monoidal category $\C$. This will of course be key to writing down the generators of the standard A-H F monoid(s) of Definition \ref{AHFMreflex-def}. \end{remark}

\begin{theorem}\label{buildingFA-thm}
Let $(\C ,\otimes , s_{\_ , \_} , a_{\_,\_,\_} , Tr_{\_ , \_}^{- })$ be a symmetric traced monoidal category, and let $N=N\otimes N\in Ob(\C)$ be a (non-unit) strictly self-similar object. We follow the conventions \& notation of Proposition \ref{notationabuse-prop}, and denote the underlying set of $\C(N,N)={\bf Int\C}((N,N),(N,N))$ by $M$, along with the two semi-monoidal monoid structures on it given by 
\begin{itemize}
\item (From the underlying traced category) $(M,\cdot,\otimes , \tau ,\sigma)$ 
\item (From the compact closed category)  $(M,\circ, \Box, T,S)$. 
\end{itemize}
Let us denote the inverse of $\tau \in (M,\cdot)$ by $\tau'\in (M,\cdot)$.
Then the standard A-H Frobenius monoid is the monoid generated by the closure of the set   
\[  \{ \tau\cdot (1 \otimes \sigma)\ ,\  (1\otimes \sigma)\cdot \tau' \ ,\  \tau \otimes \tau' \ , \ \tau'\otimes \tau \}  \]
under the composition $\circ$ and tensor $\Box$.
\end{theorem}
\begin{proof}
From Proposition \ref{canonicalsplitmerge-prop}, the split and merge arrows are given by $\Delta =  \tau\cdot (1 \otimes \sigma)$ and $\nabla = (1\otimes \sigma)\cdot \tau'$ respectively. The associator for $\_ \Box\_$ and its inverse are given by $T=\tau \otimes \tau'$ and $T'=\tau'\otimes \tau$ respectively.  These primitives then, by definition, generate the standard A-H Frobenius monoid, and satisfy the {\em Frobenius condition} 
\[ (1\Box \nabla) \circ T' \circ (\Delta\Box 1) \ = \ \Delta \circ \nabla \ = \ (\nabla \Box  1) \circ T(1\Box  \Delta) \]
along with {\em associativity} and {\em co-associativity}
\begin{itemize}
\item $\nabla\circ (1\Box  \nabla) = \nabla\circ (\nabla\Box 1)\circ T$  
\item $(1\Box \Delta)\circ \Delta=((\Delta \Box 1)\circ \Delta) \circ T$. 
\end{itemize}
\end{proof}

\subsection{Algebraic aspects}\label{algebra-sect}
The standard A-H F monoid at a  strictly reflexive object contains, by construction, all canonical associativity isomorphisms of the semi-monoidal endomorphism monoid; it therefore (by Theorem \ref{Fassoc-thm}) contains a copy of  Thompson's iconic group  $\mathcal F$.  We now demonstrate that the simple case also contains a copy of one of the most iconic structures from semigroup theory -- the bicyclic monoid.  Further, there is a very natural way in which the simple A-H F monoid may be thought of as, ``interacting copies of the bicyclic monoid, and Thompson's $\mathcal F$''.  

The following was first published in \cite{EL}, but appears previously to have been known by Clifford, Preston, \& Rees (see \cite{CH} for a historical overview).
\begin{definition}
The {\bf bicyclic monoid} $\bf B$ is the monoid with two generators, and a single relation
\[ {\mathcal B} \ = \ \langle s,r \ : \ rs = 1 \rangle \] 
(Note that this is a one-sided inverse, and certainly does not imply that $sr=1$)!
 \end{definition}
It is remarkably well-studied, and appears in a wide range of algebraic and computational settings. Its theory is very well-established; the following results may be found in, for example, \cite{MVL}.

\begin{theorem}\label{bicyclic-thm}
\ \\ \begin{enumerate}
\item $\mathcal B$ is an inverse monoid, and is isomorphic to the inverse monoid of partial injections on the natural numbers generated by the successor function and its (partially defined) inverse.
\item Given an arbitrary monoid $M$, and a pair of elements $r,s\in M$ satisfying $rs=1\neq sr$, then the submonoid generated by $\{r,s\}$ is isomorphic to $\mathcal B$.
\item The elements of $\mathcal B$ may be given a normal form as pairs of natural numbers, with composition given by, for all $(d,c),(b,a)\in \mathbb N \times \mathbb N$ 
\[ (d,c) (b,a) \ = \   \left(d + [b\monus c], [c\monus b]+a \right) \] 
where the {\bf monus} operation $\monus$ is defined by $y\monus x = \left\{ \begin{array}{lr} y-x & x\leq y \\ 0, & \mbox{otherwise.}\end{array}\right.$
\end{enumerate}
\end{theorem}

\begin{remark} It is hard to avoid seeing the bicyclic monoid itself as a categorical structure.   It is well-known that the natural numbers with the usual ordering forms a (posetal) category, and addition is a (strictly symmetric \& associative) monoidal tensor on this category, with the unit object simply being $0\in \mathbb N$.  Further, it is straightforward that $\mathbb N$ is a {\em closed} monoidal category, with the internal hom given by the above monus operation, $y\monus x = \left\{ \begin{array}{lr} y-x & x\leq y \\ 0, & \mbox{otherwise.}\end{array}\right.$

The bicyclic monoid then has all the appearance of being the result of a `dualising' construction on this monoidal closed category. It may also be relevant that $(\mathbb N , + , 0)$ is traced, with the trace given on objects by subtraction.
\end{remark}

\begin{theorem}\label{FaB-thm} Let $N$ be a (non-unit) strictly reflexive object of a compact closed category $(\C,\otimes , t_{\_,\_,\_} , s_{\_,\_} , I) $ with trivial scalars. Then $\C(N,N)$ is a semi-monoidal monoid containing a copy of the simple A-H F monoid, which we denote $\mathbb A \subseteq \C(N,N)$.  Then 
\begin{enumerate}
\item $\mathbb A$ is generated by  the closure under composition and tensor of :
\begin{itemize}
\item The associator $\alpha=t_{N,N,N}$ and its inverse $\alpha^{-1}$.
\item The split map $\Delta$ and its generalised inverse $\nabla$.
\end{itemize}
\item The associator and its inverse generate (by closure under composition \& tensor), a copy of Thompson's $\mathcal F$ within $\mathbb A$.
\item The split element $\Delta$ and its generalised inverse $\nabla$ generate (by closure under composition) a copy of the bicyclic monoid $\mathcal B$ within $\mathbb A$.\end{enumerate}
\end{theorem} 
\begin{proof}
Part 1. is simply the definition of the simple A-H F monoid at $N$. Part 2. follows directly from Theorem \ref{Fassoc-thm}.  From  Corollary \ref{onesidedinverse-corol}, the split and merge arrows of the simple A-H F monoid satisfy $\nabla\Delta=1$ and $\Delta\nabla\neq 1$, and from Part 2. of Theorem \ref{bicyclic-thm}, they therefore generate a copy of the bicyclic monoid.  
\end{proof}

\section{Concrete examples}\label{concrete-sect}
We now move on from the abstract theory of strict extensional reflexivity to concrete examples. These have been  separated out from the theoretical constructions, in order to emphasise that the theory is generally applicable, and not tied to any specific example. 

The concrete setting we now consider it that of Girard's original two Geometry of Interaction papers \cite{GOI1,GOI2} (and indeed Abramsky, Haghverdi, \& Scott's linear combinatory logic \cite{AHS}).  The starting point for this is  the (inverse) traced monoidal category of partial injections, considered as a traced monoidal subcategory of an illustrative example of Joyal, Street, and Verity.

\subsection{The category of relations, its matrix calculus, and its trace}
In \cite{JSV}, Joyal, Street, and Verity used the monoidal category of relations with disjoint union as a canonical example of a traced, but not compact closed, category. Their treatment was based on writing relations in matrix form.
\begin{definition}
	The category ${\bf Rel}$ of relations on sets has as objects all sets. An arrow $R\in {\bf Rel}(X,Y)$ is a subset $R\subseteq Y\times X$. Given an arrow $S\in {\bf Rel}(Y,Z)$, composition is given by the usual formula for relational composition,
	\[ (z,x)\in SR \ \ \mbox{ iff } \exists y\in Y \ : \ (z,y)\in S \ \mbox{ and }\ (y,x)\in R \]  
The category of relations also has a dagger $(\ )^c:{\bf Rel}^{op}\rightarrow {\bf Rel}$, given by relational converse $R^c=\{ (x,y) : (y,x)\in R\}$.

A partial function is a relation satisfying $(y,x),(z,x)\in R \Rightarrow y=z$, and a partial injection is a partial function whose converse is also a partial function. It is standard to write partial functions and partial injections in functional form, as $f(x)=y$, rather than $(y,x)\in f$. Partial functions and partial injections form wide subcategories of $\bf Rel$, denoted $\bf pFun$ and $\bf pInj$ respectively.  

The category $\bf Rel$ also has a biproduct, the disjoint union $\_ \uplus \_ : {\bf Rel}\times {\bf Rel}\rightarrow {\bf Rel}$ given by on objects by disjoint union $A\uplus B = A \times \{ 0 \} \cup  B \times \{ 1 \}$ and extended to arrows in the obvious manner.

The biproduct structure implies the existence of projection and injection arrows. For all $A,B\in  Ob({\bf Rel})$, the projection arrows are the following partial injections   
	\begin{itemize}
		\item  $\pi_0 = \{ (a,(a,0)) : a\in A \} \in {\bf Rel}(A\uplus B, A)$ 
		\item  $\pi_1 = \{ (b,(b,1)) : b\in B \} \in {\bf Rel}(A\uplus B, B)$ 
	\end{itemize}
and the injection arrows are their relational converses
	\begin{itemize}
	\item  $\iota_0 = \{ ((a,0),a) : a\in A \} \in {\bf Rel}(A,A\uplus B)$ 
	\item  $\iota_1 = \{ ((b,1),b) : b\in B \} \in {\bf Rel}(B,A\uplus B)$ 
\end{itemize}
Both $\bf pFun$ and $\bf pInj$ are closed under $\_ \uplus\_$; however it is simply a symmetric monoidal tensor on these subcategories and neither a product nor a coproduct. In all three settings, the disjoint union  has as unit object  the empty set $I=\{ \}$, and hence a trivial monoid of scalars. 

\end{definition}

\begin{remark}[Units, scalars, and  strict associativity]
	 The unit object of $\bf Rel$ is not a strict unit; rather, $A\uplus I = A \times \{ 0 \} \cup \{ \} \times \{1 \} = A\times \{ 0 \} \cong A$. 

The unique arrow of the endomorphism monoid of $I$ is the nowhere-defined function on the empty set, thus $\bf Rel$ and the distinguished monoidal subcategories described above provide good examples of unit objects with trivial scalars. The elements (i.e. arrows of the form $f:I\rightarrow X$) are similarly uninteresting; they are the nowhere-defined partial injections whose domain is the empty set. These categories will therefore illustrate how the $\bf Int$ or $\bf GoI$ construction builds a rich structure of elements and names that is essentially unrelated to that of the underlying traced category.  
 
	Note that disjoint union is associative up to canonical isomorphism, but is not strictly associative. In \cite{JSV}, Joyal, Street,\& Verity implicitly appeal to MacLane's strictification procedure for associativity \& units \cite{MCL}, and work within a suitably strictified version of $({\bf Rel},\uplus )$. As well as eliminating repeated re-bracketings, this also provides the justification for working with arbitrary (finite) matrices, as discussed below.
	
	In our setting we need to be cautious of the result of \cite{JHRS}, that we cannot simultaneously strictify associativity and self-similarity (see also Section \ref{strictselfsim-sect}). Hence, following Section \ref{strictify-sect}, we are not able to assume strict associativity in a setting where we wish to find strictly reflexive objects -- at least, in the compact closed setting. 
	
	In the following sections we therefore do not assume strict associativity, although canonical isomorphisms may occasionally be omitted for reasons of clarity. We also restrict ourselves to $(2\times 2)$ matrices, and consider larger matrices as ``matrices of matrices'', with the precise interpretation determined by the source / target objects.  
\end{remark}

\subsection{Matrices of relations}
The following is well-established, and is a corollary of the biproduct structure described above. It is also heavily used in \cite{JSV}.
\begin{proposition}\label{matrixcomp-prop}
	Given arbitrary $F\in  {\bf Rel}(A\uplus B, P\uplus Q)$, then $F$ uniquely determines \& is uniquely  determined by a $2\times 2$ matrix of relations
	$[ f_{ij}]_{i,j\in \{ 0,1\} }$  whose entries are given in terms of the projection / injection arrows by $f_{ij}=\pi_{i}F\iota_j$. it is standard to abuse notation and write $\small F=\left( \begin{array}{cc}f_{00} & f_{01} \\ f_{10} & f_{11} \end{array}\right)$. 	
	The composite of relations in matrix form is then given by the usual formula for matrix composition, with addition and multiplication interpreted by union and  relational composition respectively:  
	\[ 
	\left( \begin{array}{ccc} g_{00} & & g_{01} \\ g_{10} & &g_{11} \end{array} \right) 
	\left( \begin{array}{ccc} f_{00} & & f_{01} \\ f_{10} & & f_{11} \end{array} \right)  
	= \left( \begin{array}{ccc} g_{00}f_{00}\cup g_{01}f_{10}  & &  g_{11}f_{11}\cup g_{10}f_{01} \\
	g_{10}f_{00}\cup g_{11}f_{10} &	& g_{11}f_{11}\cup g_{10}f_{01}
	\end{array} \right)
	\] 
	This may also be drawn via the usual `summing over paths' description of matrix composition where these matrices are drawn as digraphs:
	
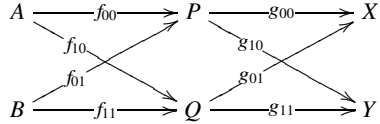
\[ 
	\xymatrix{
		A \rrto|{f_{00}} \drrto|<<<<<<<{f_{10}} & & P \rrto|{g_{00}}\drrto|<<<<<<<{g_{10}} & & X \\
		B \rrto|{f_{11}} \urrto|<<<<<<<{f_{01}} & & Q\rrto|{g_{11}} \urrto|<<<<<<<{g_{01}}&  & Y  \\
	}
	\] 
	and matrix composition interprets as `summing over paths from source to target':
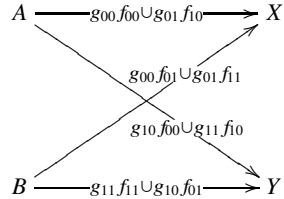
	\[
	\xymatrix{
		A \xto[0,3]|{g_{00}f_{00}\cup g_{01}f_{10}} \xto[2,3]|>>>>>>>>>>>>{g_{10}f_{00}\cup g_{11}f_{10}} & & & X  \\
		& & & \\
		B \xto[0,3]|{g_{11}f_{11}\cup g_{10}f_{01}} \xto[-2,3]|>>>>>>>>>>>>{g_{00}f_{01}\cup g_{01}f_{11}} & & & Y \\
	}
	\]
\end{proposition}
\begin{proof} {\em [Outline]} This is very well-established, and the formula for matrix composition may be derived from the observation that the composite of a projection and an injection  acts like a Kronecker delta, so $ \pi_j\iota_i \ = \ \left\{ \begin{array}{lcr} 0 &  & i\neq j \\
	1 & & i=j \end{array}\right.$
\end{proof}

The category $\bf Rel$ was used as an illustration of a traced symmetric monoidal category in \cite{JSV}, and the structure of the resulting compact closed category $\bf Int(Rel)$ was also described in detail. 
The following, from this reference, is key:
\begin{theorem}\label{reltrace-thm}
	The category $({\bf Rel},\uplus)$ is traced, with the trace defined in terms of the reflexive transitive closure of relations. Given $F=\left( \begin{array}{cc} f_{00} & f_{01} \\ f_{10} & f_{11} \end{array}\right) : X\uplus U \rightarrow Y\uplus U$, then 
	\[ Tr_{X,Y}^U(F)\ =\  f_{00} \cup f_{01}\left(  \bigcup_{j=0}^\infty f_{11}^j \right) f_{10} \] 
\end{theorem}
\begin{proof} The proof that this operation satisfies the axioms of Definition \ref{Trace-def} is given in \cite{JSV}, and relies heavily on the properties of the Kleene star (i.e. reflexive transitive closure) of relations in endomorphism monoids.
	\end{proof}

\begin{remark} In \cite{SA96}, concrete examples of categorical traces were described as belonging to one of two classes : either `particle-style' (based on iteration or feedback), or `wave-style' (based on fixed-points). The above trace on $({\bf Rel},\uplus)$ is the canonical example of a particle-style trace.
	\end{remark}  

\section{A matrix formalism for partial injections}
As they are monoidal subcategories of $({\bf Rel},\uplus )$, both $({\bf pFun},\uplus)$ and $({\bf pInj},\uplus)$ admit matrix representations of arrows.  However, unlike $\bf Rel$, their homsets are not closed under arbitrary unions, so some care is needed when using a matrix formalism. The case of partial functions was covered by Manes \& Arbib in their study of algebraic program semantics \cite{MA}, and based on this, necessary and sufficient conditions for a matrix of partial injections to represent a partial injection was given in \cite{PHD}.
 
\begin{proposition}\label{pInjmat-prop}
	Given $X,V,Y,U\in Ob({\bf pInj})$, the matrix representations of partial injections in ${\bf pInj}(X\uplus V , Y\uplus U)$ are precisely those matrices $\small \left( \begin{array}{cc} a & b \\ c & d \end{array} \right)$ whose entries are partial injections such that the following diagrams 
	commute:
	\[ \xymatrix{ 
		Y \ar[d]_{c^\ddagger}	& X \ar[l]_{a} \ar[d]^{b}\ar[dl]|{0_{XV}}				 	& & & Y & X\ar[l]_{a} \\  
		V  						& U \ar[l]^{d^\ddagger} 										& & & Y\ar [u]^{c} & U \ar[u]_{b^\ddagger} \ar[l]^{d^\ddagger} \ar[ul]|{0_{UY}} \\  
	}
	\]
\end{proposition}
\begin{proof}[Outline] {\em This was first proved in \cite{PHD}. The key points of this proof are given below, in order to give some insight into the structure of matrix representations within $\bf pInj$.}\\ 
The starting point is the observation that, given a family of partial injections $\{ f_j\}_{j\in J} \in {\bf pInj}(X,Y)$, their union $\bigcup_{j\in J} f_j\in {\bf Rel}(X,Y)$ is a partial injection iff $f_i^{\ddagger}f_j$ and $f_ig_j^{\ddagger}$ are idempotent, for all $i,j\in J$.
An important special case is when $f_i^{\ddagger}f_j=0_X$ and $f_if_j^{\ddagger}=0_Y$, so $f_i\cap f_j = \emptyset$, and hence the union $\bigcup_{j\in J}f_j$ is trivially a partial injection. 

Now consider the above matrix of partial injections. The commutativity of the above two diagrams is equivalent (up to the appropriate inclusions) to this condition. In the other direction, the projection / injection arrows impose this condition when writing down the matrix form of a partial injection. 
\end{proof}

\begin{definition}\label{rooksquare-def}
The above condition is sometimes known as the {\bf rook matrix condition}, since it states that elements in the same row of a matrix have disjoint images, and elements in the same column have disjoint domains. This description is more useful for matrix calculations within $\bf pInj$, but the characterisation as commuting squares is more useful when working with the compact closed category $\bf Int(pInj)$. 

We therefore refer to a square $\vcenter{\vbox{\small 
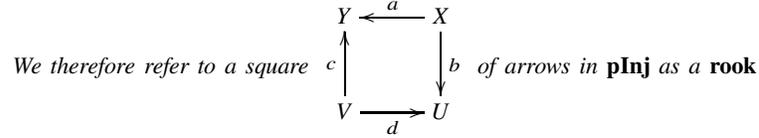
\xymatrix{ Y & X \ar[l]_{a} \ar[d]^b \\ V \ar[u]^c \ar[r]_d & U }} }$ of arrows in $\bf pInj$ as a {\bf rook square} when the two diagrams of Proposition \ref{pInjmat-prop} above commute.
\end{definition}

As a corollary of the above characterisation, it was shown in \cite{PHD} that $({\bf pInj},\uplus)$ is closed under the trace described in Theorem \ref{reltrace-thm}.
\begin{theorem}
	The symmetric monoidal category $({\bf pInj} ,\uplus)$ is traced, with the trace given by 
	\[ Tr^U_{X,Y}\left( \begin{array}{cc} f_{00} & f_{01} \\ f_{10} & f_{11} \end{array}\right) \  = \  f_{00} \cup \bigcup_{j=0}^\infty \left(f_{01}f_{11}^jf_{10}\right) \]
\end{theorem}
\begin{proof} This was first given in \cite{PHD}, but also proved independently in \cite{HA,HS,AHS}.
\end{proof}

\begin{remark} 	Unlike the case of $\bf Rel$, it is inaccurate to write the above trace in $\bf pInj$ in terms of the Kleene star, as $f_{00} \cup f_{01}\left( \bigcup_{j=0}^\infty f_{11}^j\right) f_{10}$. The reflexive transitive closure $\left( \bigcup_{j=0}^\infty f_{11}^j\right)$ is in general simply a relation, and not a partial injection. 
\end{remark}

\section{The compact closed category $\bf Int(pInj)$}
We now describe the compact closed category that results from applying Joyal, Street, \& Verity's $\bf Int$ construction to the traced category $({\bf pInj},\uplus)$ of partial injections. We emphasise that this is simply the explicit description of $\bf Int(Rel)$, as found in \cite{JSV}, restricted to the traced symmetric monoidal subcategory of partial injections.

\begin{definition}\label{IPdef}
	Plugging in the traced monoidal category $(\bf pInj , \uplus)$ into the $\bf Int$ construction of \cite{JSV} gives the following:
	
	\begin{description}
		\item[{\bf Objects}] 
		These are pairs of objects of $\bf pInj$ (i.e. pairs of sets). 
		\item [{\bf Arrows}]
		Homsets are given by ${\bf Int(pInj)}((X,U),(Y,V))={\bf pInj}(X\uplus V,Y\uplus U)$. Arrows of ${\bf pInj}(X\uplus V,Y\uplus U)$ are given as $(2\times 2)$ matrices of partial injections satisfying the rook matrix condition such as $\left( \begin{array}{cc} f_{00} & f_{01} \\ f_{10} & f_{11} \end{array}\right) \in {\bf Rel}(X\uplus V , Y \uplus U)$.  Following [Joyal et al. 96], we draw the corresponding arrow of ${\bf Int(pInj)}((X,U),(Y,V))$ as a (planar) graphical 4-tuple of partial injections satisfying the rook square condition (Definition \ref{rooksquare-def}), as follows:
		\scalebox{0.8}{$\vcenter{
				\vbox{ \small
					\xymatrix{  Y & X \ar[l]_{f_{00}} \ar[d]^{f_{10}} \\ V \ar[u]^{f_{01}} \ar[r]_{f_{11}} & U  }
				} 
			}$.}
		\item [{\bf Composition}] To compose such digraphs, we simply glue then along their common edge, followed by taking the union over all paths with identical source \& target: 
		\begin{center}
			\scalebox{0.8}{
				\xymatrix{ 
					Z 												&& Y \ar[ll]_{g_{00}}\ar@/_9pt/[dd]|{g_{10}}			&&X \ar[ll]_{f_{00}} \ar@/_9pt/[dd]|{f_{10}}				&& && 	Z																																																				&&											&& X	 \ar[llll]_{\bigcup_{j=0}^\infty g_{00}\left( f_{01}g_{10}\right)^j f_{00}}  \ar@/_9pt/[dd]|{ f_{10} \cup \ \bigcup_{j=0}^\infty f_{11}g_{10} \left( f_{01}g_{10}^j \right) f_{00}}		\\
					&&														&&															&& && 																																																					&&											&& 																		\\		
					W \ar@/_9pt/[uu]|{g_{01}} \ar[rr]_{g_{11}}		&& V\ar@/_9pt/[uu]|{f_{01}} \ar[rr]_{f_{11}}			&&U 														&& && 	W \ar@/_9pt/[uu] |{ g_{01} \cup \ \bigcup_{j=0}^\infty g_{00}f_{01} \left( g_{10}f_{01}^j \right) g_{11}}  \ar[rrrr]_ {\bigcup_{j=0}^\infty f_{11}\left( g_{01}f_{10}\right)^j g_{11}} 							&&											&& U																		\\
				}
			} 
		\end{center}
		\item [{\bf Monoidal tensor}]
		The tensor $\_ \Box \_$ is defined on objects by $(X,U)\Box(X',U')=(X\uplus X',U'\uplus U)$ and on arrows by:
		\begin{center}
			\scalebox{0.8}{$
				\vcenter{
					\vbox{ \small
						\xymatrix{  Y & X \ar[l]_{a} \ar[d]^{c} \\ V \ar[u]^{b} \ar[r]_{d} & U  }
					} 
				} \ {\large{\Box}} \ 
				\vcenter{
					\vbox{ \small
						\xymatrix{  Y' & X' \ar[l]_{a'} \ar[d]^{c'} \\ V' \ar[u]^{b'} \ar[r]_{d'} & U'  }
					} 
				}
				\ = \ 
				\vcenter{
					\vbox{ \small
						\xymatrix@R+1.2pc@C+2.3pc{  Y\uplus Y' & X\uplus X' \ar[l]|{\tiny \left(\begin{array}{cc} a & 0 \\ 0 & a' \end{array}\right) } \ar[d]|{\tiny \left(\begin{array}{cc} 0 & c' \\ c & 0 \end{array}\right) } \\ V'\uplus V \ar[u]|{\tiny \left(\begin{array}{cc} 0 & b \\ b' & 0 \end{array}\right)}  \ar[r]|{\tiny \left(\begin{array}{cc} d' & 0 \\ 0 & d \end{array}\right) } & U'\uplus U  }
					} 
				}
				$}
		\end{center}
		\item [{\bf The unit object}] This is simply the object $( \emptyset,\emptyset)$. 
		\item [{\bf The dual}] 
		The contravariant dual $( \ )^*:{\bf Int(pInj)}^{op}\rightarrow {\bf Int(pInj)}$ is defined on objects by $(X,U)^*=(U,X)$, and on arrows by 
		\scalebox{0.8}{$
			\left( \vcenter{
				\vbox{ \small
					\xymatrix{  Y & X \ar[l]_{a} \ar[d]^{c} \\ V \ar[u]^{b} \ar[r]_{d} & U  }
				} 
			} 
			\right)^*\ \  {\large = } \ \  
			\vcenter{
				\vbox{ \small
					\xymatrix{  U & V \ar[l]_{d} \ar[d]^{b} \\ X \ar[u]^{c} \ar[r]_{a} & Y  }
				} 
			}
			$}. 
		\item [{\bf The compact closed structure}] The unit and co-unit arrows for the compact closed structure, 
		\[ \eta_{(X,U)}: (\emptyset,\emptyset) \rightarrow (X,U)\Box (X,U)^* \ \mbox{ and } \ \epsilon_{(X,U)}: (X,U)^*\Box(X,U)\rightarrow (\emptyset,\emptyset) \] are given by, respectively:
		
		\[
		\vcenter{
			\vbox{ 
				\xymatrix{  
					X\uplus U & \emptyset \ar[l]_{0} \ar[d]^{0} \\ 
					U\uplus X \ar[u]^{{\tiny \left(\begin{array}{cc} 0 & 1 \\ 1 & 0 \end{array}\right)}} \ar[r]_{0} & \emptyset  }
			} 
		} 
		\ \ \ \mbox{ and } \ \  \
		\vcenter{
			\vbox{ 
				\xymatrix{  
			\emptyset  & U \uplus X \ar[l]_{0} \ar[d]^{\tiny \left( \begin{array}{cc} 0 & 1 \\ 1 & 0\end{array} \right) } \\ 
			\emptyset \ar[u]^{0} \ar[r]_{0} & X\uplus U  }
			} 
		}
		\] 
	\end{description}
\end{definition} 

\begin{remark}\label{IntGoIsquares-rem}[On the formul\ae\ for composition]The seemingly intricate formula for the composition of two rook squares arises in a very natural way from the diagrammatic representation of matrix composition, described in Proposition \ref{matrixcomp-prop}. 

Taking the `summing over paths' intuition seriously, the formula for the trace of $({\bf Rel},\uplus)$, and hence $({\bf pInj},\uplus)$, arises by introducing a feedback loop to the matrix representation of an arrow, and again summing over paths as shown :
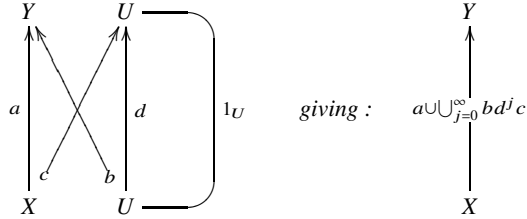
\[ 
\xymatrix{
Y							&	U	\ar@{-} `r[dr] `[dd]^{1_U} [dd]&	&					& Y		\\
							&										&	& \mbox{ giving : }	&	\\
X\ar[uu]^a\ar[uur]|<<<c		&	U\ar[uul]|<<<b \ar[uu]_d			&	&					& X\ar[uu]|{a \cup \bigcup_{j=0}^\infty bd^jc}	\\
							&										&	&					&   \\
					}
\]
	In $\bf pInj$, both composition and trace are given by `summing over paths' constructions. Bringing these together gives the formula, and formalism, for composition in $\bf Int(pInj)$;  consider arrows
	$F : (X,U)\rightarrow (Y,V)$ and $G : (Y,V)\rightarrow (Z,W)$ determined by rook matrices
	$\tiny \left( \begin{array}{cc} f_{00} & f_{01} \\ f_{10} & f_{11} \end{array}\right) $
	and $\tiny \left( \begin{array}{cc} g_{00} & g_{01} \\ g_{10} & g_{11} \end{array}\right)$ respectively. 
	The composite $GF\in  {\bf Int(pInj)}((X,U),(Z,W))$ is then given by summing over paths in the following diagram :
	\begin{center}
		\scalebox{0.7}{\small
			
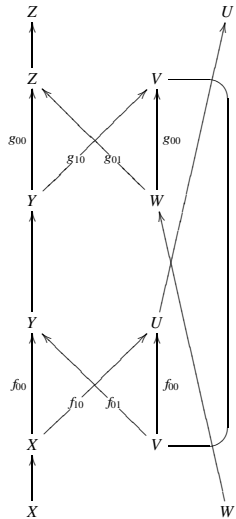
\xymatrix{
				Z												&							&												&	U				&&	\\
				Z \ar[u]										&							&	V \ar@{-} `r[dr] `[dddddd] [dddddd]			&					&&	\\
				&							&												&					&& \\
				Y \ar[uu]^{g_{00}} \ar[uurr]|(0.35){g_{10}}	&							&	W \ar[uu]_{g_{00}} \ar[uull]|(0.35){g_{01}}	&					&&	\\
				&							&												&					&&	\\			
				Y  \ar[uu]										&							&	U \ar[uuuuur]								&					&&	\\
				&							&												&					&&	\\
				X \ar[uu]^{f_{00}} \ar[uurr]|(0.35){f_{10}}	&							& V	\ar[uu]_{f_{00}} \ar[uull]|(0.35){f_{01}}	&					&&	\\
				X \ar[u]										&							&												&	W \ar[uuuuul]	&&	\\
			} 
		}
	\end{center}
	Untangling this(!) but keeping the overall directed graph topology gives precisely the `planar squares' formalism and composition described above. Thus, the rook squares formalism arises from taking the digraph representation of matrices of partial injections and re-drawing it in a planar manner. Similarly, the composition of IP is the natural `summing over paths' operation that -- unlike matrix composition -- preserves planarity\footnote{An interesting open question is how much of the $\bf Int$ construction (at least in the `particle-style' setting) may be thought of as `imposing planarity' on some matrix calculus.  We refer to \cite{SA08} for further connections between planarity and (particle-style) compact closure.}.
	
	Readers familiar with the conventions of \cite{SA96}, as well as those of \cite{JSV} may wish to verify that Abramsky's composition within $\bf GoI(pInj)$ may be drawn in similar terms, and corresponds to {\em vertical} rather than {\em horizontal} pasting of rook squares, as illustrated below : 
		
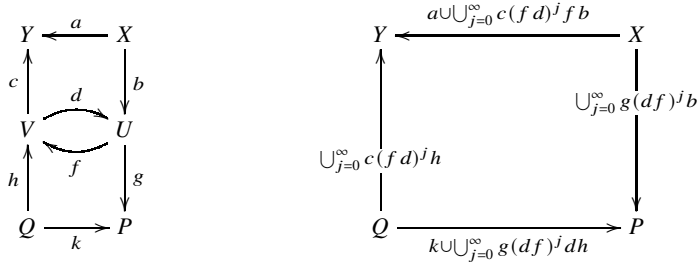
\[
	\xymatrix{ 
		Y 								& X\ar [l]_{a} \ar [d]^b   								&&&  Y & 	&	&X \ar [lll]_{a \cup \bigcup_{j=0}^\infty c(fd)^jfb} \ar [dd]|(0.35){\bigcup_{j=0}^\infty g(df)^jb} \\
		V \ar [u]^c \ar @/^8pt/[r]^d	& U	\ar @/^8pt/[l]^f \ar [d]^g							&&&	 	& 	&	&\\
		Q \ar [u]^h \ar [r]_k			& P														&&&	Q \ar [rrr]_{k \cup \bigcup_{j=0}^\infty g(df)^jdh } \ar [uu]|(0.35){\bigcup_{j=0}^\infty c(fd)^jh} 	&	&  &P	
	}
	\]
Thus, at least at endomorphism monoids of self-dual objects (such as the reflexive, or strictly reflexive objects we discuss), it is reasonable to consider the compositions of \cite{JSV} and \cite{SA96} as distinct, but interacting, operations on the same underlying set (see Section \ref{fd-sect}). 
\end{remark}

\section{Self-similarity, and strict self-similarity in $\bf pInj$}\label{Nselfsim-sect}
We now exhibit self-similar objects of $({\bf pInj},\uplus)$, and demonstrate the semi-monoidal strictification of this self-similarity, as a first step towards exhibiting reflexive objects of $\bf Int(pInj)$, and strictifying this reflexivity.  

It is easy to find self-similar objects of $({\bf pInj},\uplus)$; any countably infinite set $D$ will suffice, and appropriate bijections between $D$ and $D\uplus D$ are well-illustrated by the Hilbert's familiar parable of the Grand Hotel. Following the conventions of \cite{GOI1,GOI2}, we will take the natural numbers as our canonical example, with self-similarity exhibited by the usual (bijective) Cantor pairing. 
\begin{definition}\label{Cantor-def} The {\bf Cantor pairing} is the bijection $\code : \mathbb N \uplus \mathbb N \rightarrow \mathbb N$ given by $\code  (n,i)=2n+1$. We denote its inverse by $\decode : \mathbb N \rightarrow \mathbb N\uplus \mathbb N$; explicitly,
		\[ \decode (n) = \left\{ \begin{array}{lcr} \left( \frac{n}{2} , 0 \right) & & n \mbox{ even,} \\ & & \\ \left(\frac{n-1}{2} , 1 \right) & & n \mbox{ odd.}\end{array} \right. \]
		
		This is closely related to the {\bf dynamical algebra} of  \cite{DR,GOI1,GOI2}, which  
		is the (inverse) submonoid of ${\bf pInj}(\N,\N)$ generated by the following partial injections
		\[ p(n)= \left\{ \begin{array}{lcr} 
			\frac{n}{2} & & n \ \mbox{ even } \\ & &  \\
			\mbox{undefined} & &n \ \mbox{ odd }
           \end{array}\right.	
           \ \ \ \ , \ \ \ \ q(n)= \left\{ \begin{array}{lcr} 
           \mbox{undefined} & & n \ \mbox{ even } \\ & &  \\
           \frac{n-1}{2} & & n \ \mbox{ odd }
           \end{array}\right.	
           \]
           together with their generalised inverses $p^\ddagger(n)=2n$ and $q^\ddagger (n)=2n+1$.  These satisfy the following key conditions :  
           	\begin{enumerate}
           	\item $pp^\ddagger = 1 = qq^\ddagger$
           	\item $pq^\ddagger = 0 =qp^\ddagger$
           	\item $p^\ddagger p \cup q^\ddagger q = 1$.
           \end{enumerate}
           \end{definition}
       \begin{remark} Conditions 1. and 2. above are the defining relations for the (two-generator) polycyclic monoid of \cite{NP}, and condition 3. is a natural condition on concrete representations of polycyclic monoids. This observation was made in \cite{PHD,MVL} and polycyclic monoids form a significant and active research area in inverse semigroup theory generally.
       	\end{remark}
       
       \begin{proposition} The generators of the dynamical algebra, and their generalised inverses, arise as composites of the Cantor pairing and the canonical projection / injection arrows associated with the disjoint union, as 
       	\[ p=\code\iota_0\ \ ,\ \ q=\code\iota_1 \ \ ,\ \ p^\ddagger= \pi_0\decode\ \ ,\ \ q^\ddagger=\pi_1\decode \]
       	\end{proposition}
       \begin{proof} This may be verified almost instantly by direct calculation. It is also a special case of a more general connection between polycyclic monoids and categorical projections / injections described in \cite{TAC}.
       	\end{proof} 
       
       \begin{lemma}\label{isoendo_lem} The endomorphism monoids of $\N$ and $\N\uplus \N$ in $\bf pInj$ are isomorphic. 
\end{lemma} 
\begin{proof} The isomorphisms between them are given by conjugation by the code / decode arrows, as  : 
\begin{itemize}
	\item $\code (\_ ) \decode :  {\bf pInj} (\N\uplus \N,\N\uplus \N) \rightarrow {\bf pInj} (\N\uplus \N,\N\uplus \N)$ 
	\item $\decode ( \_ )\code :  {\bf pInj} (\N,\N) \rightarrow  {\bf pInj} (\N\uplus \N,\N\uplus \N)$
\end{itemize}
\end{proof}

The interpretation of the dynamical algebra as composites of the Cantor pairing and projections / injections then allows us to make the link with matrix representations of arrows.

\begin{corollary}\label{dynamicisoendo-corol}
	Let $\small \left( \begin{array}{cc}a & b \\ c & d \end{array}\right) \in {\bf pInj}(\N\uplus \N,\N\uplus \N)$ be a rook matrix. Then
	\[ \code  \left( \begin{array}{cc}a & b \\ c & d \end{array}\right) \decode = p^\ddagger ap \cup p^\ddagger b q \cup q^\ddagger cp \cup q^\ddagger d q \in {\bf pInj}(\N,\N) \]
	Conversely, given $f\in {\bf pInj}(\N,\N)$, then its inverse image under the monoid isomorphism described above is the following rook matrix:
	\[ \decode f \code = \left( \begin{array}{cc} pfp^\ddagger & pfq^\ddagger \\ qfp^\ddagger & qfq^\ddagger \end{array} \right)\in {\bf pInj}(\N\uplus \N , \N \uplus \N) \]
\end{corollary}

\begin{remark}\label{rooksquarebij-rem} As well as the above monoid isomorphism between ${\bf pInj}(\N,\N)$ and ${\bf pInj}(\N\uplus \N,\N\uplus \N)$, there is also, by construction, a bijection between the rook squares of ${\bf Int(pInj)}((\N,\N),(\N,\N))$, and the rook matrices of ${\bf pInj}(\N\uplus \N,\N\uplus \N)$. 
	Thus, every arrow $f\in {\bf pInj}(\N,\N)$ uniquely determines and is determined by a rook square, with the correspondence given by $f\mapsto$ \scalebox{0.8}{$\vcenter{
			\vbox{ \small
				\xymatrix{  \N & \N \ar[l]_{pfp^\ddagger} \ar[d]^{qfp^\ddagger} \\ \N \ar[u]^{pfq^\ddagger} \ar[r]_{qfq^\ddagger} & \N  }
			} 
		}$.}
\end{remark}

\subsection{Strictifying self-similarity within $\bf pInj$}

We first establish some notation \& preliminary results :
\begin{definition}
	Let us denote by $({\bf pNat},\uplus,\emptyset)$ the full symmetric monoidal subcategory of $\bf pInj$ generated by the natural numbers, together with disjoint union. Note that $(\bf pNat ,\uplus)$ is also traced, with trace given by the trace of $({\bf pInj},\uplus)$.
\end{definition}

We now give an explicit description of the semi-monoidal strictification of self-similarity within $\bf pNat$.  We take the code and decode arrows $\code: \N \uplus \N \rightarrow \N$ and $\decode : N \rightarrow \N \uplus \N$ to be as given in Definition \ref{Cantor-def}, and similarly for the dynamical algebra.

Our starting point is a symmetric semi-monoidal tensor on the endomorphism monoid of $\N$; this is well-known as Girard's representation of multiplicative conjunction within \cite{GOI1,GOI2}.
\begin{definition}
	Given $f,g\in {\bf pNat}(\N,\N)$, we define their tensor $f\star g\in {\bf pNat}(\N,\N)$ by $f\star g \ = \ \code(f\uplus g)\decode$. 
	This may alternatively and equivalently be given in terms of the dynamical algebra, as $f\star g = p^\ddagger f p \cup q^\ddagger g q$. Explicitly, 
	\[ (f\star g)(n) \ = \ \left\{ \begin{array}{lcr} 2f\left(\frac{n}{2}\right) & & n \ \mbox{ even, and } \frac{n}{2}\in dom(f) \\ & \ \ \ \ & \\
	2g\left(\frac{n-1}{2}\right)+1 & & n \ \mbox{ odd, and }\  \frac{n-1}{2} \in dom(g) \\ && \\
	\bot & & \mbox{otherwise}
	\end{array}\right.
	\]
	The symmetry and associativity isomorphisms $\sigma,\tau\in {\bf pNat}(\N,\N)$ for this tensor  are given by $\sigma(n)\ = \left\{ \begin{array}{lcr} n+1 & \ \ & n \mbox{ even} \\ & & \\ n-1 & \ \ & n \mbox{ odd} \end{array}\right.$ and $ \tau(n)  \ = \ \left\{ \begin{array}{lcr} 2n & \ \ & n \ (mod \ 2)=0 \\  & \ \ & \\
	n+1 & \ \ & n \ (mod \ 4)=1 \\ 
	  & \ \ & \\
	\frac{n-1}{2} & \ \ & n \ (mod \ 4)=3 \\ 
	 \end{array}\right.$
	 
	 These also may be given explicitly, in terms of the dynamical algebra, as
	 \[ \sigma = p^\ddagger q \cup q^\ddagger p \ \ \mbox{ and } \ \ \tau=\left( p^\ddagger\right)^2 p \cup p^\ddagger q^\ddagger pq \cup q^\ddagger q^2 \]
	\end{definition}

\begin{theorem} The above operation $\_ \star \_$ is indeed a symmetric semi-monoidal tensor on the endomorphism monoid of $\N$, with associativity and symmetry maps as given above.
	\end{theorem}
\begin{proof} This is well-established \cite{PHD,TAC}, and the interpretation as a semi-monoidal strictification of self-similarity is given in \cite{JHRS}. Explicit elementary arithmetic proofs of MacLane's pentagon and hexagon conditions are also given in \cite{RC}.
	\end{proof}

Now consider the semi-monoidal category $({\bf pNat}_{-I},\uplus)$ given by the `forgetting the unit' functor of Definition \ref{units_stuff-def}. As well as the above tensor, the strictification procedure of \cite{JHRS} gives, as described in Section \ref{SSstrict-sect},  a semi-monoidal equivalence of categories between $({\bf pNat}_{-I},\uplus)$ and $({\bf pNat(\N,\N)},\star)$. This proceeds as follows :
	
	\begin{definition}\label{Phifunctor-def}
		For all objects $X\in Ob({\bf pNat}_{-I})$, we define mutually inverse isomorphisms $C_X: X\rightarrow \N$ and $D_X:\N\rightarrow X$ inductively, by $C_\N=1_\N$, and for all $A,B\in Ob ({\bf pNat}_{-I})$,
	\[ C_{A\uplus B}=\code(C_A\uplus C_B) \ \ , \ \ D_{A\uplus B}=(D_A\uplus D_B)\decode \]
	We then define a semi-monoidal functor $\Phi : ({\bf pNat}_{-I},\uplus)  \rightarrow ({\bf pNat(\N,\N)},\star)$ by :
	\begin{description}
		\item[Objects] $\Phi(X)=\N$ for all $X\in Ob({\bf pNat})$.
		\item[Arrows] $\Phi(f)=C_B f D_A\in {\bf pNat(\N,\N)}$, for all $f\in {\bf pNat}(A,B)$
	\end{description}
\end{definition}

\begin{theorem} The functor $\Phi : ({\bf pNat}_{-I},\uplus)  \rightarrow ({\bf pNat(\N,\N)},\star)$ is a semi-monoidal equivalence of categories.
\end{theorem}
\begin{proof}
	This is immediate from the strictification procedure of \cite{JHRS} and heavily prefigured (albeit without the interpretation as a semi-monoidal equivalence of categories) in \cite{PHD,TAC}.
\end{proof}

We may now adjoint a strict unit to the above categories, as described in Definition \ref{units_stuff-def}. This results in the the following categories :
\begin{itemize}
	\item $({\bf pNat(\N,\N)},\star)_{+I}$, the above semi-monoidal category with a strict unit adjoined.
	\item $(\bf pNat,\uplus)_{-\eafw}$, the de-elemented version of $(\bf pNat,\uplus)$.
\end{itemize}

\begin{corollary}\label{tracedmonoid-corol}  \ \\
	\begin{enumerate}
		\item $(\bf pNat,\uplus)_{-\eafw}$ is monoidally equivalent to $({\bf pNat(\N,\N)},\star)_{+I}$.
		\item Both $(\bf pNat,\uplus)_{-\eafw}$ and $({\bf pNat(\N,\N)},\star)_{+I}$ are traced.
	\end{enumerate}
\end{corollary}	
\begin{proof} \ \\
	\begin{enumerate}
		\item 	This follows immediately from Corollary \ref{eafw-equiv-corol}, and the monoidal equivalence is given by a trivial extension of the $\Phi$ semi-monoidal functor of Definition \ref{Phifunctor-def} above to categories with strict units adjoined.
		\item This follows from Proposition \ref{keyequivalences-prop}. The trace on $(\bf pNat,\uplus)_{-\eafw}$ is simply that of $({\bf pNat },\uplus)$, restricted to non-element arrows. The trace of $({\bf pNat}(\N,\N),\star)_{+I}$ is given by, for all $f\in {\bf pNat}(\N,\N)$,
		\begin{itemize}
			\item $Tr_{I,I}^I (1_I) = 1_I$
			\item $Tr_{\N,\N}^I (f) = f$
			\item $Tr_{\N,\N}^\N (f) = f_{00}\cup \bigcup_{j=0}^\infty f_{01}f^j_{11}f_{10}$ where the components $f_{\_,\_}$ are as given in Corollary \ref{dynamicisoendo-corol}, as
			\[ f_{00} = pfp^\dagger \ \ , \ \  f_{01}= pfq^\ddagger \ \ , \ \ f_{10}= qfp^\ddagger \ \ , \ \ f_{11}= qfq^\ddagger  \]
		\end{itemize}
	\end{enumerate}
\end{proof}

\section{From self-similarity in $\bf pInj$ to (strict) reflexivity in $\bf Int(pInj)$}
As an immediate consequence of the self-similarity of the natural numbers, we may give self-dual self-similar objects in a compact closed category : 
\begin{lemma}\label{NNselfsim-lem}
	The (strictly) self-dual object $(N,N)=(N,N)^*\in Ob({\bf Int(pInj)})$ is self-similar.
	\end{lemma}
\begin{proof}
	From Corollary \ref{intselfsim-corol}, as $\mathbb N \in Ob({\bf pInj})$ is self-similar, so is $(\mathbb N,\mathbb N)\in Ob({\bf Int(pInj)})$, with the code / decode arrows for $(\N,\N)$ given by the following rook squares : 
\[ 
\small 
\xymatrix{ 
\N \uplus \N 						& \N \ar[l]_{\decode} \ar[d]^0	& & & \N & \N\uplus \N\ar[l]_{\code} \ar[d]^0 \\  
\N \uplus \N \ar[u]^0	\ar[r]_{\code}	& \N 						& & & \N \ar [u]^{0} \ar[r]_{\decode} & \N\uplus \N    \\  
}
\]
	\end{proof}

\begin{corollary}\label{NNreflex-corol}
	$(\N,\N)\in Ob({\bf Int(pInj)})$ is an extensionally reflexive object i.e. it is isomorphic to its own internal hom., so  $[(\N,\N)\rightarrow(\N,\N)] \cong (\N,\N)$.
	\end{corollary}
\begin{proof} This is immediate from the characterisation of extensionally reflexive objects given in Lemma \ref{reflexchar-lem}.
	\end{proof}

 Our stated aim is to provide concrete examples of how reflexivity may be strictified in a compact closed category --- how we may give a compact closed subcategory containing the specified reflexive object, together with a monoidal equivalence to another compact closed category in which this reflexivity is exhibited by identity arrows. We do so by applying the abstract procedures laid out in Section \ref{strictify-sect} to the above natural example of an extensionally reflexive object in $\bf Int(pInj)$.

The following, although individually straightforward, will prove powerful : 
\begin{proposition}\label{pNatresults-prop}\ \\
	\begin{enumerate}
	\item $\bf Int(pNat)$ is a compact closed category where all objects $(X,U)\in Ob({\bf Int(pNat)})$ satisfying $X\ncong \emptyset \ncong U$ are isomorphic.
	\item Let us denote by $\bf IpN$ the full monoidal subcategory of $\bf Int(pNat)$ generated by the self-dual object $(\N,\N)$. Then $\bf IpN$  is a compact closed category where all non-unit objects are isomorphic.
	\end{enumerate}
	\end{proposition}
\begin{proof}\ \\
	\begin{enumerate}
	\item By construction, arbitrary non-unit objects $X,Y,U,V\in Ob({\bf pNat})$ are all isomorphic. Let us fix  isomorphisms $\phi\in {\bf pNat}(X,Y)$ and $\psi\in {\bf pNat}(U,V)$. Then the following rook squares give an isomorphism in $\bf Int(pNat)((X,U),(Y,V))$, together with its inverse : 
	\[ 
	\small 
	\xymatrix{ 
		Y						& X  \ar[l]_{\phi} \ar[d]^0	& & & X & Y\ar[l]_{\phi^{-1}} \ar[d]^0 \\  
		V\ar[u]^0	\ar[r]_{\psi^{-1}}	& U						& & & U \ar [u]^{0} \ar[r]_{\psi} & V    \\  
	}
	\]
	\item Note that point 1. above does {\em not} imply that all non-unit objects of $\bf Int(pNat)$ are isomorphic; counterexamples are provided by $(X,I)$ and $(I,X)$, where $X\neq \emptyset \in Ob({\bf pNat})$. However, by construction, the full monoidal subcategory $\bf IpN$ generated by $(\N,\N)$ is closed under tensor and dual, and by point 3., all non-unit objects are isomorphic.  As it is a full subcategory, it also contains the relevant unit / co-unit arrows.
\end{enumerate}
\end{proof}
Following the program laid out in above, we are of course moving towards a monoidal equivalence between the category $\bf IpN$ of part 4. above, and a compact closed category with a single non-unit object.

\subsection{A monoidal strictification of extensional reflexivity}
Corollary \ref{tracedmonoid-corol} above gives us precisely what we require for the strictification of reflexivity at a reflexive object of a compact closed category. We have a two-object traced monoidal category (i.e. $({\bf pNat}(\N,\N),\star)_{+I}$) that is monoidally equivalent to a small subcategory of a (de-elemented version of) $\bf pInj$, generated by a reflexive object. 

\begin{remark}\label{G-rem}
The next step is to apply the $\bf Int$ construction to this two-object traced monoidal category; 
this will result in a four-object compact closed category that contains a strictly reflexive object, $(\N,\N)$. We then consider the two-object compact closed monoidal subcategory generated by this distinguished object.
\end{remark}

\begin{definition} \label{G-def}  As our notation is in danger of becoming unwieldy at this point, let us simply denote by $\mathfrak G$ the two-object compact closed category resulting from the steps of Remark \ref{G-rem} above.
	
Expanding out the definitions results in the following:
\begin{description}
	\item[{\bf Objects}] $Ob(\mathfrak G) = \{ (I,I) , (\N,\N) \}$
	\item[{\bf Hom-sets}] By construction, all homsets (excluding the endomorphism monoid of the units) are equal : 
	\begin{description}
		\item[{\bf Scalars}] $\mathfrak G ((I,I),(I,I)) = \{ 1_I \} $
		\item[{\bf Elements}] $\mathfrak G ((I,I),(\N,\N)) = {\bf pNat}(I\star \N,\N\star I) = {\bf pNat}(\N,\N)$
		\item[{\bf Co-Elements}] $\mathfrak G ((\N,\N),(I,I)) = {\bf pNat}(\N\star I ,I \star \N) = {\bf pNat}(\N,\N)$
		\item[{\bf Endomorphisms}] $\mathfrak G ((\N,\N),(\N,\N)) = {\bf pNat}(\N\star \N,\N\star \N) = {\bf pNat}(\N,\N)$
		\end{description}
	 It will be convenient to describe arrows of homsets using the correspondence between members  of ${\bf pNat}(\N,\N)$ and rook squares over ${\bf pNat}(\N,\N)$ of Remark \ref{rooksquarebij-rem}. Given the same endomorphism $f\in {\bf pNat}(\N,\N)$, we draw it in different ways as a rook square\footnote{Note the use of a formal zero arrow, to denote an empty homset, in the rook squares for elements / co-elements. This is simply a notational convenience -- we are not claiming the existence of a zero arrow between the formal unit object and other objects in the same category. Thus we have not quite arrived back in the category $\bf pInj$, although the difference between an empty homset, and a homset containing the nowhere-defined function on the empty set, is subtle!} in different homsets of $\mathfrak G$.
	\begin{description}
		\item[{\bf Elements}] In $\mathfrak G ((I,I),(\N,\N))$, we draw $f$ as 
			$\vcenter{
			\vbox{ \small
		 \xymatrix{  \N & I \ar[l]_0  \ar[d]^{1_I} \\ \N \ar[u]^{f} \ar[r]_{0 } & I  } } }$
		\item[{\bf Co-Elements}] In $\mathfrak G ((\N,\N),(I,I))$, we draw $f$ as
			$\vcenter{
			\vbox{ \small \xymatrix{  I & \N \ar[l]_{0} \ar[d]^{f} \\ I \ar[u]^{1_I} \ar[r]_{0} & \N  } } }$
		\item[{\bf Endomorphisms}] In $\mathfrak G ((\N,\N),(\N,\N))$, we draw $f$ as 
			$\vcenter{
			\vbox{ \small \xymatrix{  \N & \N \ar[l]_{pfp^\ddagger} \ar[d]^{qfp^\ddagger} \\ \N \ar[u]^{pfq^\ddagger} \ar[r]_{qfq^\ddagger} & \N  }
		} }$\\
		and refer to Remark \ref{rooksquarebij-rem} for the observation that this rook square uniquely determines and is determined by $f\in {\bf pNat}(\N,\N)$. 
	\end{description}
	\item[{\bf Composition}] This is given by the `pasting rook squares and summing over paths' described in Definition \ref{IPdef}. 
	\item[{\bf The tensor}] Tensors with the unit object are defined by strictness, so $(I,I)\Box \_$ and $\_ \Box (I,I)$ are identity functors. At the non-unit object, we have  $(\N,\N)\Box(\N,\N)=(\N,\N)$, and for arrows, the tensor is defined on rook square representations, as \\ 
	\scalebox{0.9}{$
		\vcenter{
			\vbox{ \small
				\xymatrix{  \N & \N \ar[l]_{f_{00}} \ar[d]^{f_{10}} \\ \N \ar[u]^{f_{01}} \ar[r]_{f_{11}} & \N  }
			} 
		} \ {\large{\Box}} \ 
		\vcenter{
			\vbox{ 
				\xymatrix{  \N & \N \ar[l]_{g_{00}} \ar[d]^{g_{10}} \\ \N \ar[u]^{g_{01}} \ar[r]_{g_{11}} & \N  }
			} 
		}
		\ \ \  = \ \ \ 
		\vcenter{
			\vbox{ 
				\xymatrix@R+1.2pc@C+2.3pc{  \N & &  \N \ar[ll]|{\small p^\ddagger f_{00}p\cup q^\ddagger g_{00}q } \ar[d]|{\tiny q^\ddagger g_{10}p \cup p^\ddagger f_{10}q } \\  \N \ar[u]|{\tiny p^\ddagger f_{01}q \cup q^\ddagger g_{01}p }  \ar[rr]|{\tiny p^\ddagger g_{11}p \cup q^\ddagger f_{11}q  } & & \N  }
			} 
		}
		$}
\item [{\bf The canonical isomorphisms}] \ \\
\begin{description}
	\item[{\bf Associativity}] From Lemma \ref{Intcanonisos-lem}, the associator $T$ for the above tensor is  $\tau\star \tau^{-1}$, where $\star$ is the tensor of ${\bf pNat}(\N,\N)_{+I}$ and $\left( p^\ddagger\right)^2 p \cup p^\ddagger q^\ddagger pq \cup q^\ddagger q^2$ is the corresponding associator.  In rook square notation, this gives \scalebox{0.8}{$\vcenter{
			\vbox{ \small \xymatrix{  \N & \N \ar[l]_{\tau} \ar[d]^{0} \\ \N \ar[u]^{0} \ar[r]_{\tau^{-1}} & \N  }
		} }$}\\
	\item[{\bf Symmetry}] Also from Lemma \ref{Intcanonisos-lem}, the commutativity isomorphism for $\_ \Box\_$ is given by $\sigma\star \sigma$, where $\sigma=q^\dagger p \cup p^\dagger q$  is the commutativity isomorphism for $\_ \star \_$. In rook square notation this is simply  \scalebox{0.8}{$\vcenter{
		\vbox{ \small \xymatrix{  \N & \N \ar[l]_{\sigma} \ar[d]^{0} \\ \N \ar[u]^{0} \ar[r]_{\sigma} & \N  }
	} }$}\\
\item [{\bf The unit object}] By construction, we have a strict unit object, $(I,I)$.
\end{description}
\item [{\bf The compact closed structure}] \ \\
\begin{description}
\item [{\bf The dual}] The dual of $\mathfrak G$ is a dagger, which we nevertheless write as $( \ )^*$, to avoid confusion with the generalised inverse of $\bf pInj$. Thus, on objects, $\mathbb N^*=\mathbb N$, and on arrows it is given by Definition \ref{Int-def}. Explicitly, given an arrow in $\mathfrak G ((X,U),(Y,V))$ represented as a rook square, its dual is given by  
\scalebox{0.8}{$
	\left( \vcenter{
		\vbox{ \small
			\xymatrix{  Y & X \ar[l]_{a} \ar[d]^{c} \\ V \ar[u]^{b} \ar[r]_{d} & U  }
		} 
	} 
	\right)^*\ \  {\large = } \ \  
	\vcenter{
		\vbox{ \small
			\xymatrix{  U & V \ar[l]_{d} \ar[d]^{b} \\ X \ar[u]^{c} \ar[r]_{a} & Y  }
		} 
	}
	$} 
\item [{\bf The unit and co-unit}] Recall that within $\mathfrak G$, 
\[ (\N,\N)\Box (\N,\N)^* \ = \ (\N,\N) \ = \ (\N,\N)^* \Box (\N,\N) \] 
We then have the two distinguished arrows for the compact closed structure,  
$\eta : (I,I) \rightarrow (\N,\N)$ and $\epsilon : (\N,\N) \rightarrow (I,I)$ 
given by, respectively :
\[
\vcenter{
	\vbox{ 
		\xymatrix{  
			\N & I \ar[l]_{0} \ar[d]^{0} \\ 
			\N \ar[u]^{\sigma} \ar[r]_{0} & I  }
	} 
} 
\ \ \ \mbox{ and } \ \  \
\vcenter{
	\vbox{ 
		\xymatrix{  
			I  & \N \ar[l]_{0} \ar[d]^{\sigma} \\ 
			I \ar[u]^{0} \ar[r]_{0} & \N  }
	} 
}
\]
where $\sigma=p^\ddagger q \cup q^\ddagger p$ is the symmetry map for $\_ \star \_$, the tensor of the underlying traced monoidal category.
\end{description}
\end{description}
	  
\end{definition}

\begin{remark}
The above two-object compact closed category provides an example of the strictification of extensional reflexivity described in abstract terms in the first half of this paper; we have already seen that the following key properties are satisfied : 
	\begin{enumerate}
		\item 	The compact closed category $(\mathfrak G, \Box , (\ )^*,(I,I))$ is monoidally equivalent to the compact closed subcategory of $\bf Int(pInj)$ monoidally generated by $(\N,\N)\in Ob({\bf Int(pInj)})$.   
		\item	The object $(\N,\N)\in Ob({\bf Int(pInj)})$ is extensionally reflexive.
		\item	The unique non-unit object $(\N,\N)\in Ob(\mathfrak G)$ is strictly extensionally reflexive.
		\end{enumerate}

	\end{remark}

\section{Algebraic aspects}\label{algebraexam-sect}
Having established concrete examples of the relevant abstract category theory, we move on to considering the  algebraic aspects  -- precisely, the embeddings of Thompson's $\mathcal F$ associated with coherence for associativity, and the simple Abramsky-Heunen Frobenius monoid associated with strict reflexivity in a compact closed category.

\begin{definition}
Following Point \ref{samehomsets-pt} of Section \ref{strictdisc-sect}, we observe that the endomorphism monoids $\mathfrak G((\N,\N),(\N,\N))$  and $\bf pInj (\N,\N)$ have the same underlying set. Let us denote this by $\mathcal H$. It will be convenient to use rook matrix notation for elements of $\Hi$, and write a partial injection $f$ as $\left( \begin{array}{cc} p f p^\ddagger & p f q^\ddagger \\ q f p^\ddagger & q f q^\ddagger \end{array} \right)$; entirely equivalently, we say that $\left( \begin{array}{cc} 
	a & b \\ c & d \end{array}
\right)$ is simply shorthand for the partial injection $
p^\ddagger ap \cup p^\ddagger bq \cup q^\ddagger cp \cup q^\ddagger dq$.
(As a notational device, we will also denote composition within $\bf pInj$ itself simply by concatenation).

Let us denote the composition on $\Hi$ arising from $\bf pInj$ as 
$\_ \cdot \_ : \Hi \times \Hi \rightarrow \Hi$, and that arising from $\mathfrak G$ as $\_ \circ \_ : \Hi\times \Hi \rightarrow \Hi$.  Explicitly, these are given by : 
\begin{itemize}
	\item  $
			\left( 
				\begin{array}{cc} 	e & f \\ 
									g & h 
				\end{array} 
			\right) 
			\cdot 
			\left( 
				\begin{array}{cc} 
						a & b \\ 
						c & d 
				\end{array} 
			\right) 
			= 
			\left(  \begin{array}{cc} ea \cup fc &  eb \cup fd  \\ ga \cup hc &  gb \cup hd \end{array} \right)$.
	\item $
	\left( 
	\begin{array}{cc} 	e & f \\ 
	g & h 
	\end{array} 
	\right) 
	\circ 
	\left( 
	\begin{array}{cc} 
	a & b \\ 
	c & d 
	\end{array} 
	\right) 
	= 
	\left( \begin{array}{cc}
	    \bigcup_{j=0}^\infty e\left( bg\right)^j a 
	    &
	    f \cup \ \bigcup_{j=0}^\infty eb \left( gb^j \right) h
	    \\
	    c \cup \ \bigcup_{j=0}^\infty dg \left( bg^j \right) a
	    &
	    \bigcup_{j=0}^\infty d\left( fc\right)^j h
	    \end{array} \right)$
	\end{itemize}

By construction, we have two additional operations $\_ \star \_  , \_ \Box \_ : \Hi\times \Hi \rightarrow \Hi$ that are symmetric semi-monoidal tensors for  $(\Hi,\cdot)$ and $(\Hi,\circ)$ respectively. These may be given explicitly, by : 
\begin{itemize}
	\item$
	\left( 
	\begin{array}{cc} 	e & f \\ 
	g & h 
	\end{array} 
	\right) 
	\star
	\left( 
	\begin{array}{cc} 
	a & b \\ 
	c & d 
	\end{array} 
	\right) 
	= 
	\left(  \begin{array}{cc}p^\ddagger ep \cup p^\ddagger fq \cup q^\ddagger gp \cup q^\ddagger hq &  0  \\ 0  & p^\ddagger a p \cup p^\ddagger b q \cup q^\ddagger c p \cup q^\ddagger d q \end{array} \right)$.
	\item $
	\left( 
	\begin{array}{cc} 	e & f \\ 
	g & h 
	\end{array} 
	\right) 
	\Box 
	\left( 
	\begin{array}{cc} 
	a & b \\ 
	c & d 
	\end{array} 
	\right) 
	= 
	\left(  \begin{array}{cc} p^\ddagger a p \cup q^\ddagger e q &  p^\ddagger b q \cup q^\ddagger f p   \\ p^\ddagger c q \cup q^\ddagger g p  &  p^\ddagger h p \cup q^\ddagger d q  \end{array} \right)$.
	\end{itemize}
The symmetry and associativity isomorphisms for the semi-monoidal monoid $(\Hi,\cdot,\star)$ are given by, respectively, \[ \sigma = \left( \begin{array}{cc} 0 & 1 \\ 1 & 0 \end{array} \right) \ \ \mbox{ and } \ \ \tau = \left( \begin{array}{cc} p^\ddagger & q^\ddagger p \\ 0  & q\end{array} \right) \]
The inverse of $\tau$, with respect to the composition $\_ \cdot \_$ is given explicitly by $\tau' = \left(\begin{array}{cc} p & 0 \\ p^\ddagger q  & q^\ddagger \end{array}\right)$, and $\sigma$ is self-inverse w.r.t. the same composition.

Similarly, the associativity and symmetry isomorphisms for the semi-monoidal monoid $(\Hi,\circ,\Box)$ are given by, respectively, $S = \left( \begin{array}{cc} \sigma & 0 \\ 0 & \sigma \end{array} \right)$ and $T = \left( \begin{array}{cc} \tau & 0 \\ 0 & \tau'\end{array} \right)$. 

Expanding out these definitions in terms of the dynamical algebra, we get $S = \left( \begin{array}{cc} p^\ddagger q \cup q^\ddagger p & 0 \\ 0 & p^\ddagger p \cup q^\ddagger p \end{array} \right)$  and 
\[ T = \left( \begin{array}{cc} \left( p^\ddagger \right)^2p \cup p^\ddagger q^\ddagger pq \cup q^\ddagger q^2 & 0 \\ 0 & p^\ddagger p^2 \cup q^\ddagger p^\ddagger qp \cup \left(q^\ddagger\right)^2 q \end{array} \right) 
\]

$S$ is then its own inverse, w.r.t. both compositions on $\Hi$. The inverse of $T$ (again, w.r.t. both compositions) is given by 
\[ T'= \left( \begin{array}{cc}
p^\ddagger p^2 \cup q^\ddagger p^\ddagger qp \cup \left(q^\ddagger\right)^2 q & 0 \\
	& \left( p^\ddagger \right)^2p \cup p^\ddagger q^\ddagger pq \cup q^\ddagger q^2 
\end{array} \right) \]
\end{definition}

The following identities are then immediate from both the algebraic description, and the details of the $\bf Int$ construction.
\begin{lemma}The above distinct associativity and symmetry elements are related as follows : 
\begin{itemize}
	\item $S=\sigma \star \sigma = \sigma \cdot S \cdot \sigma$
	\item $T=\tau \star \tau'$
	\item  $T'=\sigma \cdot T \cdot \sigma$
	\end{itemize}
\end{lemma}
\begin{proof}
	These follow, simply by construction.
	\end{proof}

The above structure, consisting of two distinct monoid compositions and two distinct semi-monoidal structures, along with non-trivial interactions between them, then provides examples of the algebraic structures discussed in Sections \ref{F-sect} and \ref{Fmonoid-sect}. 

\begin{theorem}\ 
	\\
	\begin{enumerate}
	\item The submonoid of $(\Hi,\cdot)$ generated by $\{ \tau , \tau' , 1\star \tau , 1\star \tau' \}$ is a group isomorphic to Thompson's group $\mathcal F$, and is also generated by the closure of $\{ \tau ,\tau'\}$ under the composition $\cdot$ and the tensor $\star$.
	\item The submonoid of $(\Hi,\circ)$ generated by $\{ T , T',  1\Box T , 1\Box T'  \}$ is again a subgroup isomorphic to Thompson's group $\mathcal F$, and is also generated by the closure of $\{ T,T'\}$ under the composition $\circ$ and the tensor $\Box$.
	\item The elements 
	\begin{itemize}
		\item $\Delta = \tau \cdot (1 \star \sigma)$
		\item $\nabla = (1\star \sigma) \cdot \tau'$
		\end{itemize}
	satisfy $\Delta \cdot \nabla = 1 = \nabla \cdot \Delta$, but 
	\[ \nabla \circ \Delta = 1 \ \ \mbox{ and } \ \ \Delta \circ \nabla \neq 1 \]
	Hence $\{ \Delta ,\nabla \} $ generates a copy of the bicyclic monoid within $(\Hi , \circ)$.
	\item The above elements satisfy  
	\begin{description}
		\item [The Frobenius condition] $(1\Box \nabla) \circ T' \circ (\Delta\Box 1)  =  \Delta \circ \nabla  =  (\nabla \Box  1) \circ T(1\Box  \Delta)$
		\item [Associativity] $\nabla\circ (1\Box  \nabla) = \nabla\circ (\nabla\Box 1)\circ T$
		\item [Co-Associativity] $(1\Box \Delta)\circ \Delta=((\Delta \Box 1)\circ \Delta) \circ T$. 
		\end{description}and hence generate a copy of the simple Abramsky-Heunen Frobenius algebra within the semi-monoidal monoid $(\Hi,\circ , \Box)$.
	\end{enumerate}
\end{theorem}
\begin{proof} 
	\ \\
	\begin{enumerate}
		\item As $(\Hi,\cdot, \star)$ is a semi-monoidal monoid, this follows directly from Theorem \ref{Fassoc-thm}.
		\item Similarly, $(\Hi,\circ, \Box)$ is a semi-monoidal monoid; this again follows directly from Theorem \ref{Fassoc-thm}.
		\item Although this is a corollary of the categorical reasoning of Part 3. of Theorem \ref{FaB-thm}, it is also perhaps the last point at which a purely algebraic proof is readily accessible, so we prove this directly, as a check that our categorical reasoning is indeed correct. 
		
		Expanding out the definition gives that 
		\[ \Delta = \tau \cdot (1\star \sigma) = 
		\left(
		\begin{array}{cc}
		p^\ddagger & q^\ddagger p \\
		0			& q 
		\end{array} 
		\right) 
		\left(
		\begin{array}{cc}
		1  & 0 \\
		0	& q^\ddagger p \cup p^\ddagger q 
		\end{array} 
		\right)
		=
		\left(
		\begin{array}{cc}
		p^\ddagger & q^\ddagger q \\
		0			& p 
		\end{array} 
		\right)
		\]
		Similarly, 
		\[ \nabla =  (1\star \sigma) \cdot \tau' = 
		\left(
		\begin{array}{cc}
		1 & 0 \\
		0			& q^\ddagger p \cup p^\ddagger q
		\end{array} 
		\right) 
		\left(
		\begin{array}{cc}
		 p & 0 \\
		p^\ddagger q 	& q^\ddagger p  
		\end{array} 
		\right)
		=
		\left(
		\begin{array}{cc}
		p & 0 \\
		q^\ddagger q		& p^\ddagger
		\end{array} 
		\right)
		\]
		Direct matrix composition, together with the defining relations of the dynamical algebra, then show that $\Delta \cdot \nabla = \left( \begin{array}{cc} 1 & 0 \\ 0 & 1 \end{array}\right)  = \nabla \cdot \Delta$.  For the composite derived from the compact closed structure, a straightforward route to calculating the composites $\nabla \circ \Delta$ and $\Delta \circ \nabla$ is given by moving to the `rook squares' formalism, and summing over paths within the following two diagrams : 
		 
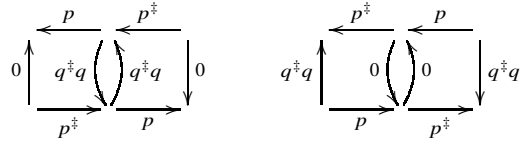
\[ 
			\xymatrix{ 
												& \ar[l]_{p}\ar@/_5pt/[d]_{q^\ddagger q}	& \ar[l]_{p^{\ddagger}} \ar[d]^{0}\\
				\ar[u]^{0} \ar[r]_{p^\ddagger}	& \ar[r]_{p} \ar@/_5pt/[u]_{q^\ddagger q}	&						\\	
			} \ \ \mbox{ \ \ \ \ \ } \ \ 
		\xymatrix{ 
			& \ar[l]_{p^\ddagger}\ar@/_5pt/[d]_{0}	& \ar[l]_{p} \ar[d]^{q^\ddagger q}\\
			\ar[u]^{q^\ddagger q} \ar[r]_{p}	& \ar[r]_{p^\ddagger } \ar@/_5pt/[u]_{0}	&						\\	
		}
			\]
			Relying on the key identities $pq^\ddagger = 0 =qp^\ddagger$ simplifies this considerably, and gives the composites as :
		\[ \nabla \circ \Delta = \left(\begin{array}{cc} 1 & 0 \\ 0 & 1 \end{array} \right) \ \ \mbox{ and } \ \ 
		\Delta \circ \nabla = \left( \begin{array}{cc} p^\ddagger p & q^\ddagger q \\ q^\ddagger q & p^\ddagger p \end{array}\right) \]
		so $\nabla\circ \Delta=1$ and $\Delta\circ \nabla \neq 1$ as required. (It is worth observing the curiosity that $(\Delta \circ \nabla)\cdot (\Delta \circ \nabla) = 1$. The categorical significance of this is currently unknown). 
		\item For this, we must simply appeal to the abstract category theory already developed, and claim it as a Corollary of Theorem \ref{buildingFA-thm}.
		\end{enumerate}
\end{proof}

\section{Future directions}\label{fd-sect}
I would like to thank Samson Abramsky for the advice that I should never give a comprehensive account of a subject, but rather leave some  `low-hanging fruit' so that other authors have the opportunity to reference me.  Although this was undoubtedly firmly tongue-in-cheek, there are nevertheless many directions that could be pursued. 
\begin{description}
	\item[\bf Categorically] It is hard to avoid the conclusion that this paper needs a good dose of the scalars. The strictification of reflexivity described in Section \ref{strictify-sect} relies on the underlying traced monoidal category having trivial scalars, and the interaction of Frobenius algebras, Thompson's $\mathcal F$ and the bicyclic monoid is also predicated on the monoid of scalars being trivial.   It is presumably possible to get rid of this requirement in both cases, and give a somewhat more sophisticated procedure that can also deal with a non-trival  monoid of scalars.  Although this work remains to be carried out, the key to it is undoubtedly the methods of \cite{SAAS} of adjoining a non-trivial commutative monoid (with involution) of scalars to a traced or compact closed category.
	\item[\bf Logically] This paper has concentrated on categorical \& algebraic aspects of the Geomery of Interaction system, rather than logical interpretations. A logical puzzle arises nevertheless; in the system of \cite{GOI1,GOI2}, there is unavoidably a simple A-H Frobenius monoid, as studied in Section \ref{Fmonoid-sect}, and given explicitly in Section \ref{algebraexam-sect}. The split / merge distinguished arrows  of a Frobenius algebra have the interpretation with quantum-mechanical systems as `fan-out'   -- a restricted form of copying that does not violate the no-cloning theorem \cite{WZ}, but does provide a great deal of computational power to quantum computational systems \cite{HS}.  The question is then simply, `to which features of (resource-sensitive) linear logic do these arrows correspond?'.  Possibly relevant is the fact that we have not, so far, given a treatment of how copying is treated within linear logic generally, or within the systems of \cite{GOI1,GOI2,AHS} in particular. Logically, this is via the $!(\_)$ modality, which -- as shown in \cite{PHD} --  appears in \cite{GOI1,GOI2} as a fixed-point functor on a semi-monoidal monoid $(M,\star)$, defined by $f\star !(f)=!(f)$, for all $f\in M$.  
	\item[\bf Algebraically] Canonical coherence arrows of semi-monoidal monoids have a habit of appearing as interesting or well-known purely algebraic structures  -- the case of Thompson's $\mathcal F$ is a prime example.   Similarly, the logicians' dynamical algebra has long been identified as not only Nivat and Perot's polycyclic monoids \cite{PHD,MVL}, but also the algebra of projections / injections for a semi-monoidal monoid \cite{PHD,TAC}.    When we consider symmetry as well as associativity arrows, we instead have to deal with Thompson's group $\mathcal V$ \cite{MVL06}.    This immediately raises the natural question of the standard or simple A-H Frobenius monoid(s).  It is entirely reasonable to expect these, considered simply as monoids rather than semi-monoidal monoids, to be well-known \& well-studied for their algebraic properties. 
	
	Theorem \ref{FaB-thm} provides some justification for the intuition that the simple A-H F monoid is a combination of Thompson's $\mathcal F$ and the bicyclic monoid, with interactions between the two determined by the Frobenius condition -- thus uniting three different iconic structures from different branches of algebra. 
	
	What is, however, missing from the above account is some subset of the A-H F monoid that generates it by closure under composition only (i.e. not closure under composition and tensor).  This would undoubtedly make it more accessible to the algebra community, which is a natural setting in which it should also be studied.
	Their identification as well-known algebra must surely be close at hand.
	\item [{\bf Notation and diagrammatics}]  The power of a good formalism for both making a subject accessible, and for developing new theory, can hardly be overestimated. This is shown by the utility of the string diagrams formalism for compact closed categories generally, and its application in the categorical quantum mechanics program in particular. 
	
	Unfortunately, it is also singularly inappropriate for working with strict reflexivity generally, and A-H Frobenius monoids in particular, due to its reliance on {\em strict} monoidal tensors. Recall from \cite{JHRS} that one cannot simultaneously have strict associativity and strict self-similarity (\& hence strict reflexivity), except in the trivial case where everything collapses to the unit object.  As the key building blocks of the A-H Frobenius monoids are (necessarily non-strict) associators, such a diagrammatic formalism is simply inapplicable. 
	
	A suitable formalism that illustrates the underlying concepts -- which in many cases are actually quite simple --  without becoming bogged down in syntax, is sorely needed.
\end{description}

\section*{Acknowledgements}
	This paper owes a large debt to Chris Heunen (Edinburgh), both for the concrete results indicated in the text, and many stimulating discussions on the topics listed of this chapter. He has also been responsible for correcting several wrong turns I had taken, which is greatly appreciated.  Thanks are due to Mark Lawson for numerous references and results on the inverse semigroup theoretic side, as well as for neat constructions that significantly simplify otherwise complex results. 	Phil Scott (Ottawa) has been very helpful, with  discussions about linear logic and the Geometry of Interaction in general, how untyped logical and computational systems should be modeled, and in particular the r\^ole of units in untyped systems.  
	I am also very grateful to Noson Yanofsky, for discussions  on the structure and properties of Thompson's group $\mathcal F$ and related structures, and their interaction with categorical coherence generally, and Hilbert Hotel style operations in particular. 
	
	Finally, thanks are due to Samson Abramsky for uncountably many reasons; attempting to list them all would be futile. 
	

\bibliographystyle{plain}
\bibliography{SVbib}
\appendix
\end{document}